\documentclass[10pt,a4paper,twoside]{article}

\usepackage{amsmath,amsfonts,amssymb,latexsym,wasysym,mathrsfs,multirow,pifont,dsfont}
\usepackage{amsthm}
\usepackage[english]{babel}
\usepackage{a4wide}
\usepackage{pstricks,graphicx}
\usepackage[scriptsize]{subfigure}
\usepackage[scriptsize,bf]{caption}
\usepackage{enumitem}
\usepackage[small,euler-digits,icomma,OT1,T1]{eulervm}
\usepackage[all,2cell]{xy} \UseAllTwocells \SilentMatrices

\theoremstyle{plain}
\newtheorem{theorem}{Theorem}[section]
\newtheorem{proposition}{Proposition}[section]
\newtheorem{lemma}{Lemma}[section]

\theoremstyle{definition}

\theoremstyle{remark}
\newtheorem{remark}{Remark}[section]

\title{Synchronization and spin-flop transitions for a mean-field XY model in random field}

\author{Francesca Collet\thanks{Dipartimento di Matematica, Alma Mater Studiorum Universit\`a di Bologna, 40126 Bologna, Italy. \emph{E-mail address:} francesca.collet@unibo.it} \and Wioletta Ruszel\thanks{Delft University of Technology, Mekelweg 4 2628 CD Delft, The Netherlands. \emph{E-mail address:} W.M.Ruszel@tudelft.nl}}

\begin{document}

\maketitle

\begin{abstract}
\noindent We characterize the phase space for the infinite volume limit of a ferromagnetic mean-field XY model in a random field pointing in one direction with two symmetric values. We determine the stationary solutions and detect possible phase transitions in the interaction strength for fixed random field intensity. We show that at low temperature magnetic ordering appears perpendicularly to the field. The latter situation corresponds to a \emph{spin-flop transition}.

\vspace{0.3cm}

\end{abstract}

\section{Introduction}
Synchronization processes are ubiquitous in nature and, in particular, are rooted in human life from the metabolic processes in our cells to the highest cognitive tasks we perform as a group of individuals. Synchronously flashing fireflies, chirping crickets, violinists playing in unison, applauding audiences, firings of neuron assemblies, pacemaker cell beats, etc. are ensembles of units able to organize spontaneously allowing order to arise starting from disordered configurations \cite{PiRoKu03,Str03}. Synchrony has attracted much interest in the last decades due to its relevance in many different contexts as biology \cite{Ban11,LiStRe97,MaM-PTsHa10}, chemistry \cite{DiBAgBaBu12}, ecology \cite{Wal13}, climatology \cite{Tat07}, sociology \cite{Lou12}, physics and engineering~\cite{ScPeFrKa,VlKoMa03,WiCoSt96}.

The fundamental features of all the mentioned examples are the presence of many objects -- each of which is oscillating in a proper sense -- and the phenomenon of mutual influence between oscillations. The attempt of modeling these complex systems therefore led to consider large families of microscopic interacting oscillators. Typically such systems are far from being (easily) tractable. The milestone, pionereed by Winfree, was to consider biological oscillators as phase oscillators, neglecting the amplitude. He discovered that a population of non-identical oscillators, coupled by all-to-all interaction terms, can exhibit a temporal analog of a phase transition producing a remarkable cooperative phenomenon \cite{Win67}. The model was subsequently refined  by Kuramoto who proposed a soluble coupled phase oscillator model, which has a simple form and undergoes a transition from an incoherent to a coherent scenario \cite{Kur75}. The original model was deterministic and its dynamics described by a system of ordinary differential equations where oscillators lie on the complete graph and are coupled through the sine of their phase difference. Since then, several variants have been extensively investigated: more general interaction function and/or network \cite{BaAl99,Dai94,SaShKu87,WaSt98}; with periodic forcing \cite{Sak02}; stochastic \cite{BeGiPa10,Gia12,Sak88} and noisy in random environment \cite{CoDaP12,GiLuPo14,Luc11}; active rotator models \cite{GiPaPePo12,JaKu14,ShKu86}; and possible combinations of these cases. It is impossible to properly account for the literature in this field, we refer to the review article \cite{ABPVRS05} for more details and further references.\\
From an application viewpoint, an interesting modification of the Kuramoto model is represented by the XY model in magnetic systems. The latter is particularly important on the two-dimensional regular lattice since it provides for a microscopic description of certain systems in solid state physics, in particular superfluid helium films \cite{KoTh73}. Further applications include superconducting films \cite{BeMoOr79}, the Coulomb gas model \cite{FrSp81}, Josephson junction arrays \cite{BeLo96} and nematic liquid crystals \cite{PaGrYu94}. \\
In this paper we consider a noisy and disordered mean-field version of the XY model. We deal with a population of $N$ planar spins, represented by angular variables lying in $[0,2\pi)$, with mean-field interaction. Each spin evolves stochastically, driven by Brownian motion, coupled with all the other particles and  subject to an external random field (quenched disorder) that accounts for effects of the environment or anisotropies in the medium where particles live. The interaction term is of the same form as Kuramoto's and therefore is of ``imitative'' type in the sense that favor particle configurations where the spins are aligned. In the limit $N \to +\infty$ this model is accurately described by an ordinary differential equation (McKean-Vlasov equation). We study the long-time behavior of this equation and obtain the full phase diagram of the stationary solutions. In particular we show that at the macroscopic level the system may undergo a transition from an incoherent to a coherent state whenever the interaction is sufficiently strong compared with the field. Indeed the cooperative behavior of the system is a result of the competition between the coupling strength and the intensity of the external field. The coupling tries to keep aligned all the elements of the population, whereas the anisotropy breaks the symmetry by imposing at each spin a privileged orientation.\\
Differently from \cite{BeGiPa10}, where the zero-field case is considered, when supercritical we do not have a continuum of coherent stationary solutions. Due to the presence of disorder, a finite number of them is selected. Namely the directions the system possibly orders are either parallel or perpendicular to the randomness and moreover only magnetic orderings orthogonal to the field are chosen by the free energy.  This transition where spins align perpendicularly to the field is called \emph{spin-flop transition} and is very special in particular on $\mathbb{Z}^2$, since Mermin-Wagner theorem forbids phase transition in $d=2$ for models with continuous symmetry \cite{MeWa66}. As far as the XY model on $\mathbb{Z}^2$ is concerned, it is worth to mention that the occurrence of a spin-flop transition has been shown in presence of an alternating external field in \cite{Enter09} and in presence of a uniaxial random field in \cite{Cra13}. \\
In comparison with the treatment in \cite{ArPe-Vi93} we consider a dichotomic quenched disorder. The results in \cite{ArPe-Vi93} are derived for an even distribution of the random field, but the techniques allow for an approximate equation of the critical curve valid only in the case of small variance field density. This requirement excludes the possibility of bi-modal distributions which indeed cover our choice.\\
For the XY model in i.i.d. dichotomic quenched disorder we are able to provide a complete and rigorous phase portrait, finding another example of spin-flop transition appearing.
\\   
 
Our paper is organized as follows. In Section~\ref{sct:definitions} we define the model and introduce the notation. In the following Section~\ref{sct:results} we present our results. Section~\ref{sct:discussion} is devoted to some conclusions and future perspectives. Finally in the Section~\ref{sc:proofs} we present the proofs.

\section{Description of the microscopic model}\label{sct:definitions}

Let $[0,2\pi)$ be the one dimensional torus. Let also $\underline{\eta}=(\eta_j)_{j=1}^N \in \{ -1,+1\}^N$ be a sequence of independent, identically distributed random variables, defined on some probability space $(\Omega, \mathcal{F}, P)$, and distributed according to $\mu$. We assume  $\mu=\frac{1}{2}(\delta_{-1}+\delta_{+1})$. \\
Given a configuration \mbox{$\underline{x}=(x_j)_{j=1}^N \in [0,2\pi)^N$} and a realization of the random environment $\underline{\eta}$, we can define the Hamiltonian $H_N(\underline{x},\underline{\eta}):  [0,2\pi)^N \times \{ -1,+1\}^N  \longrightarrow  \mathbb{R}$ as
\begin{equation}\label{HamXY:Micro}   
 H_N(\underline{x},\underline{\eta})=-\frac{\theta}{2N}\sum_{j,k=1}^N \cos(x_k-x_j) - h \sum_{j=1}^N  \eta_j  \cos x_j \,,
\end{equation}
where $x_j$ is the spin at site $j$ and $h \eta_j$, with $h >0$, is the local magnetic field associated with the same site. Let $\theta$, positive parameter,  be the coupling strength.
For given $\underline{\eta}$, the stochastic process $\underline{x}(t)=(x_j(t))_{j=1}^N$, with $t \geq 0$, is a $N$-spin system evolving as a  Markov diffusion process on $[0,2\pi)^N$, with infinitesimal generator $L_N$ acting on ${\cal{C}}^2$ functions  $f:[0,2\pi)^N \longrightarrow \mathbb{R}$ as follows:
\begin{equation}\label{InfGenXY:Micro}
L_Nf(\underline{x}) = \frac{1}{2}  \sum_{j=1}^N \frac{\partial^2 f}{\partial x_j^2} (\underline{x}) + \sum_{j=1}^N \left\{ \frac{\theta}{N} \sum_{k=1}^N \sin(x_k - x_j) - h \eta_j \sin x_j \right\} \frac{\partial f}{\partial x_j} (\underline{x}) \,.
\end{equation}
Consider the complex quantity
\begin{equation}\label{Order:Parameter}
r_{N} e^{i \Psi_N}=\frac{1}{N}\sum_{j=1}^N e^{i x_j} \, ,
\end{equation}
where $0 \leq r_N \leq 1$ gives information about the degree of alignment (magnetization) of the spins and $\Psi_N$ measures the average value they are pointing. We can reformulate the expression of the infinitesimal generator \eqref{InfGenXY:Micro} in terms of \eqref{Order:Parameter}:
\begin{equation}\label{InfGen:Micro:OrdPar}
L_N f(\underline{x}) = \frac{1}{2}  \sum_{j=1}^N \frac{\partial^2 f}{\partial x_j^2} (\underline{x}) + \sum_{j=1}^N \left\{ \theta r_N  \sin (\Psi_N - x_j) - h \eta_j \sin x_j \right\} \frac{\partial f}{\partial x_j} (\underline{x}) \, .
\end{equation}
The expressions \eqref{HamXY:Micro} and \eqref{InfGen:Micro:OrdPar} describe a system of mean field ferromagnetically coupled spins, each with its own random field and subject to diffusive dynamics. The two terms in the Hamiltonian have different effects: the first one tends to align the spins, while the second one tends to make each of them point in the direction prescribed by the magnetic field. Observe that when $\eta_j=+1$ the spin is pushed toward $0$; whereas, when $\eta_j=-1$ toward $\pi$.\\ 

For simplicity, the initial condition $\underline{x}(0)$ is such that $(x_j(0),\eta_j)_{j=1}^N$ are independent and identically distributed with law $\nu$ of the form
\begin{equation}\label{XYinitial}
\nu(dx,d\eta) = q_0(x,\eta) \mu(d\eta)dx
\end{equation}
with $\int_0^{2\pi} q_0(x,\eta) \, dx = 1$, $\mu$-almost surely. The quantity $x_j(t)$ represents the time evolution on $[0,T]$ of $j$-th spin; it is the trajectory of the single $j$-th spin in time. The space of all these paths is $\mathcal{C}[0,T]$, which is the space of the continuous function from $[0,T]$ to $[0,2\pi)$, endowed with the uniform topology.

\section{Results}\label{sct:results}

We first describe the dynamics of the process \eqref{InfGenXY:Micro}, in the limit as $N\rightarrow +\infty$, in a fixed time interval $[0,T]$. To this effect we shall derive a \emph{law of large numbers} based on a large deviation principle on the path space. Later, the possible equilibria of the limiting dynamics will be studied. 

\subsection{Limiting dynamics}

Let $(x_j [0,T])_{j=1}^N \in (\mathcal{C}[0,T])^N$ denote a path of the system in the time interval $[0,T]$, with $T$ positive and fixed. If $f: [0,2\pi) \times \{-1,+1\} \longrightarrow \mathbb{C}$, we are interested in the asymptotic (as $N \rightarrow +\infty$) behavior of \emph{empirical averages} of the form
\begin{equation}
\int fd\rho_N(t) = \frac{1}{N} \sum_{j=1}^N f(x_j(t),\eta_j)
\end{equation} 
where $(\rho_N(t))_{t\in[0,T]}$ is the \textit{flow of empirical measures}
\begin{equation}\label{empflow}
\rho_N(t) = \frac{1}{N}\sum_{j=1}^N \delta_{(x_j(t),\eta_j)}.
\end{equation}
Notice that by choosing $f(x,\eta)=e^{ix}$ we obtain the order parameter \eqref{Order:Parameter}. We may think of $\rho_N:= (\rho_N(t))_{t \in [0,T]}$ as a continuous function taking values in $\mathcal{M}_1([0,2\pi) \times \{-1,+1\})$, the space of probability measures on $[0,2\pi) \times \{-1,+1\}$ endowed with the weak convergence topology, and the related Prokhorov metric, that we denote by $d_P(\, \cdot \, , \, \cdot \, )$.\\

The first result we state concerns the dynamics of the flow of empirical measures. It is a special case of what is shown in  \cite{ArPe-Vi93,DaPdHo96,denHolla00}, so the proof is omitted. We need some more notation. For a given \mbox{$q: [0,2\pi) \times \{-1,+1\} \longrightarrow \mathbb{R}$}, we introduce the linear operator $L_q$, acting on $f: [0,2\pi) \times \{-1,+1\} \longrightarrow \mathbb{R}$ as follows:
\begin{equation}\label{Operator:MKV}
L_q f(x,\eta)= \frac{1}{2}\frac{\partial^2 f}{\partial x^2}(x,\eta) - \frac{\partial}{\partial x} \left\{ \left[ \theta r_{q} \sin(\Psi_{q} - x) - h \eta \sin x \right] f(x,\eta) \right\},
\end{equation}
where
\[
r_{q} \, e^{i\Psi_{q}} :=  \int_{\{-1,+1\}} \int_0^{2\pi} e^{ix} \, q(x,\eta) \, dx \, \mu(d\eta).
\]
Given $\underline{\eta} \in \{-1,+1\}^N$, we denote by $\mathcal{P}_N^{\underline{\eta}}$ the distribution on $(\mathcal{C}[0,T])^N$ of the Markov process with generator \eqref{InfGenXY:Micro} and initial distribution $\lambda$. We also denote by 
\[
\mathcal{P}_N\left( d \underline{x}[0,T], d\underline{\eta} \right) := \mathcal{P}_N^{\underline{\eta}} \left(d \underline{x}[0,T]\right)  \mu^{\otimes N} \left( d\underline{\eta} \right)
\]
the joint law of the process and the environment. A large deviation principle for $\mathcal{P}_N$ allows to characterize the unique limit of the sequence $\{\rho_N(\cdot)\}_{N \geq 1}$ and, in particular, makes possible to provide a Fokker-Planck equation useful to describe the time evolution of the limiting probability measure. 

\begin{theorem}
\label{thm:MKV}
The nonlinear McKean-Vlasov equation
\begin{equation}\label{MKV:Equations}
    \left\{
    \begin{array}{cccr}
        \dfrac{\partial q_t}{\partial t}  (x, \eta) & = & L_{q_t} \, q_t (x, \eta) \\
        \\
        q_0 (x,\eta) &  & \mbox{given in \eqref{XYinitial}} & \\
    \end{array}
    \right.
\end{equation}
admits a unique solution in $\mathcal{C}^1 \left[[0,T], L^1(dx \otimes \mu) \right]$, and $q_t(\cdot,\eta)$ is a probability density on $[0,2\pi)$, for $\mu$-almost every $\eta$ and every $t>0$. Moreover, for every $\varepsilon >0$ there exists $C(\varepsilon) >0 $ such that
\[
\mathcal{P}_N \left( \sup_{t \in [0,T]} d_P ( \rho_N(t), q_t) > \varepsilon \right) \leq e^{-C(\varepsilon) N}
\]
for $N$ sufficiently large, where, by abuse of notations, we identify $q_t$ with the probability $q_t(x,\eta) \mu(d\eta)dx$ on $[0,2\pi) \times \{-1,+1\}$.
\end{theorem}

In other words Theorem~\ref{thm:MKV} states that under $\mathcal{P}_N$, for every $t \in [0,T]$, the sequence of empirical measures $\left\{ \rho_N (t) \right\}_{N \geq 1}$ converges weakly, as $N\rightarrow +\infty$, to a measure $\rho_t$ admitting density $q_t(x,\eta) \mu(d\eta) dx$.

\subsection{Phase diagram for the mean field limit}

Equation \eqref{MKV:Equations} describes the infinite-volume dynamics of the system governed by generator \eqref{InfGenXY:Micro}. We are interested in the detection of $t$-stationary solutions and possible phase transitions. We recall that being $t$-stationary solution for \eqref{MKV:Equations} means to satisfy $L_{q} q \equiv 0$.\\
Next result gives a characterization of stationary solutions of \eqref{MKV:Equations}.

\begin{proposition}\label{prop:stat:sols}
Let $q:[0,2\pi) \times \{-1,+1\} \longrightarrow \mathbb{R}$, such that $q(x, \cdot)$ is measurable and $q(\cdot,\eta)$ is a probability on $[0,2\pi)$. Then $q$ is a stationary solution of \eqref{MKV:Equations} if and only if it is of the form
\begin{equation}\label{StatSol:MKV}
q(x,\eta)= [Z(\eta)]^{-1}  \cdot \exp \left\{ 2\theta r \cos (\Psi -x) + 2h \eta \cos x \right\} \,,
\end{equation}
where $Z(\eta)$ is a normalizing factor and $(r,\Psi)$ satisfies the self-consistency relation
\begin{equation}\label{Order:Parameter:Stat}
r \, e^{i \Psi} = \int_{\{-1,+1\}} \int_0^{2\pi}  e^{ix} \, q(x,\eta) \, dx \, \mu(d\eta) \, .
\end{equation}
\end{proposition}

There is a one-to-one correspondence between equilibrium distributions \eqref{StatSol:MKV}  and solutions of the self-consistency equation \eqref{Order:Parameter:Stat}. Therefore, our analysis reduces to the study of the self-consistency relation that corresponds to a bi-dimensional \emph{fixed point problem} of the form
\[
\begin{pmatrix}
r \\
\Psi
\end{pmatrix}
=
F
\begin{pmatrix}
r \\
\Psi
\end{pmatrix}.
\]
The quantities $r$ and $\Psi$ are the macroscopic counterpart of $r_N$ and $\Psi_N$ in \eqref{Order:Parameter} and still are indicators of global coherence. Solutions with $r=0$ are called \textit{paramagnetic}, those with $r>0$ are called \textit{ferromagnetic}. 

\begin{remark}
If $r=0$ the stationary distribution \eqref{StatSol:MKV} is given by
\begin{equation}\label{q*0}
q^{(0)}(x,\eta) := [Z(\eta)]^{-1}  \cdot \exp \left\{ 2h \eta \cos x \right\} \,,
\end{equation}
where $Z(\eta)$ is a normalizing factor.
\end{remark}

\begin{proposition}\label{prop:r=0}
The pair $(0, \Psi_0)$, with arbitrary $\Psi_0 \in [0,2\pi)$, is solution of \eqref{Order:Parameter:Stat} for all values of the parameters.
\end{proposition}

The next proposition shows that ferromagnetic solutions for \eqref{Order:Parameter:Stat} appear only with specific values of average position $\Psi$. In those cases $r$ may be implicitly characterized in terms of first kind modified Bessel functions
\[
I_v (y) = \frac{1}{2\pi} \int_0^{2 \pi} \cos (v \alpha) \exp \left\{ y \cos \alpha \right\} \, d\alpha
\]
of orders $v=0$ and $1$. Indeed,

\begin{proposition}\label{integralr}
The self-consistency relation \eqref{Order:Parameter:Stat} admits solutions $(r_+,\Psi_+)$, with $r_+>0$, if and only if 
\[
\Psi_+ \in \left\{ 0, \frac{\pi}{2}, \pi, \frac{3\pi}{2} \right\}.
\] 
Moreover, $r_+$ has to satisfy
\begin{equation}\label{eqBessel}
 r =
 \begin{cases} 
 \dfrac{1}{2} \left[ \dfrac{I_1(2(h+\theta r))}{I_0(2(h+\theta r))} - \dfrac{I_1(2 ( h-\theta r ))}{I_0(2 ( h-\theta r ))} \right] & \text{ if } \quad \Psi_+ \in \{0,\pi\} \\
 &\\
\dfrac{\theta r}{\sqrt{h^2+\theta^2r^2}} \dfrac{I_1(2\sqrt{h^2+\theta^2r^2})}{I_0(2\sqrt{h^2+\theta^2r^2 })} & \text{ if } \quad \Psi_+ \in \left\{ \dfrac{\pi}{2}, \dfrac{3\pi}{2} \right\},
\end{cases}
\end{equation}
where $I_v(\cdot)$ denotes the first kind modified Bessel function of order $v$.
\end{proposition}

We will state now our main theorem. It is concerned with the investigation of under what conditions ferromagnetic solutions for \eqref{Order:Parameter:Stat} may occur. In particular, it is shown that the system undergoes several phase transitions depending on the parameters.

\begin{theorem}\label{thm:ph:dia}
Set
\[
\theta_1 (h) := \frac{1}{2} \left[ \int_0^{2\pi} \sin^2 (x) q^{(0)} (x,+1) dx \right]^{-1} \,,
\]
\[
\theta_2 (h) := \frac{1}{2} \left[ \int_0^{2\pi} \cos^2 (x) q^{(0)} (x,+1) \, dx - \left( \int_0^{2\pi}\cos(x) q^{(0)} (x,+1) \, dx \right)^2 \right]^{-1}
\]
and, moreover, let $\bar{h}$ be the value of $h$ such that
\begin{equation}\label{h:bar}
\frac{d^3}{dr^3} \left[ \dfrac{I_1(2(h+\theta r))}{I_0(2(h+\theta r))} - \dfrac{I_1(2 ( h-\theta r ))}{I_0(2 ( h-\theta r ))} \right] \Bigg\vert_{r=0}   = 0 \,.
\end{equation}
Then,
\begin{itemize}
\item[$\bullet$]
for $\theta \leq \theta_1 (h)$, then there is no ferromagnetic solution;\\
\item[$\bullet$]
for $h \leq \bar{h}$ and $\theta_1 (h) < \theta \leq \theta_2 (h)$, there exist ferromagnetic solutions $\left( r_+,\frac{\pi}{2} \right)$ and $\left( r_+, \frac{3\pi}{2} \right)$;\\
\item[$\bullet$]
for $h > \bar{h}$ and $\theta_1 (h) < \theta < \theta_2 (h)$, there exists a further value $\theta_\star (h)$, with $\theta_1 (h) < \theta_\star (h) < \theta_2 (h)$, such that
\begin{itemize}
\item
if $\theta_1 (h) < \theta \leq \theta_\star (h)$, there exist ferromagnetic solutions $\left( r_+,\frac{\pi}{2} \right)$ and $\left( r_+, \frac{3\pi}{2} \right)$;
\item
if $\theta_\star (h) < \theta \leq \theta_2 (h)$, in addition to $\left( r_+,\frac{\pi}{2} \right)$ and $\left( r_+, \frac{3\pi}{2} \right)$, two further pairs of ferromagnetic solutions arise: $\left( \bar{r}_+,0 \right)$, $\left( \bar{r}_+,\pi \right)$ and $\left( \hat{r}_+,0 \right)$, $\left( \hat{r}_+,\pi \right)$, with $\bar{r}_+ \neq \hat{r}_+$.\\
\end{itemize}
\item[$\bullet$]
for $\theta > \theta_2 (h)$, there exist ferromagnetic solutions $\left( r_+, \frac{\pi}{2} \right)$, $\left( r_+,\frac{3\pi}{2} \right)$ and $\left( \bar{r}_+,0 \right)$, $\left( \bar{r}_+,\pi \right)$. 
\end{itemize}
The values $r_+$, $\bar{r}_+$ and $\hat{r}_+$ depend on the parameters $\theta$ and $h$; therefore, they vary according to the phase we are considering.
\end{theorem}


The rich scenario depicted in Theorem~\ref{thm:ph:dia} can be qualitatively summarized in the phase portrait presented in Fig.~\ref{fig:ph:dia}.

\begin{figure}[h!]
\centering
\includegraphics[scale=1]{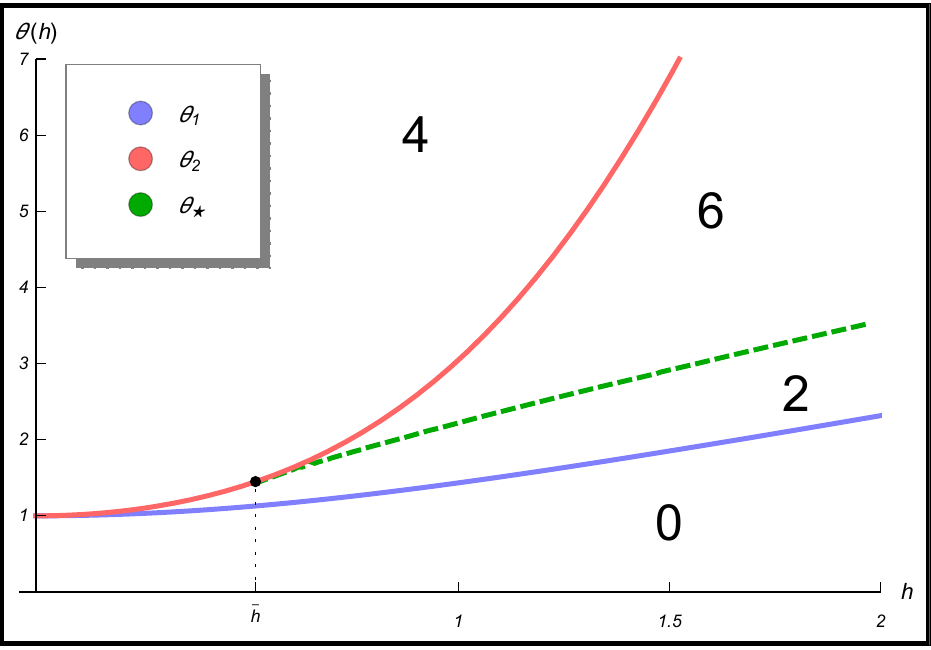}
\caption{Picture of the phase space $(h,\theta)$. Each region represents a phase with as many ferromagnetic solutions of \eqref{Order:Parameter:Stat} as indicated by the numerical label. The value $\bar{h}$ is implicitly defined by \eqref{h:bar}. The lower blue curve is $\theta_1$, whereas the upper red one corresponds to $\theta_2$. The dashed green separation line is $\theta_\star$ and is obtained numerically. Indeed, the latter is defined by a tangency relation. More hints about this curve will be given in the proof of Theorem~\ref{thm:ph:dia} in Section~\ref{sc:proofs}.
}
\label{fig:ph:dia}
\end{figure}

\section{Discussion and future perspectives}\label{sct:discussion}

The paper is concerned with the study of the phase diagram for a mean-field planar XY  model with external field. We first determined the $N \to + \infty$ limiting distribution on the path space by deriving an appropriate law of large numbers based on a large deviation principle. Then we analyzed the long-time behavior of the system. Indeed we were able to characterize the equilibria as solutions of a fixed point equation and in turn to detect several phase transitions.\\ 
Two basic mechanisms play a relevant role. On the one hand, a ferromagnetic-type interaction tends to align the spins by decreasing their phase difference. On the other, the presence of a magnetic field tends to separate the population in two different groups, pointing in opposite orientations. Clearly there is a competition between interaction and field intensity. In particular, above a critical threshold for the interaction strength a phase transition from a paramagnetic to a ferromagnetic state occurs.  More precisely we have shown that, whenever the interaction is sufficiently strong, a net magnetization appears spontaneously with average spin displacement either perpendicular or parallel to the direction fixed by the field. We then obtained $2$, $4$ or $6$ different possible ferromagnetic solutions. The richness of the phase diagram is due to the addition of a dichotomic random environment. 

\subsection{Simulations, free energy and stability}

If we integrate the model numerically, how does $\left( r_N(t), \Psi_N(t) \right)$ evolve? For concreteness, suppose we fix the field intensity $h$ and vary the coupling $\theta$. Simulations show that for all $\theta$ less than the threshold $\theta_1$, the spins act as if they were uncoupled: their values become  distributed around $0$~or~$\pi$, as prescribed by the magnetic field, starting from any initial condition (see Fig.~\ref{fig:hist:1}). 

\begin{figure}[h!]
\centering
\subfigure{\includegraphics[width=.3\textwidth]{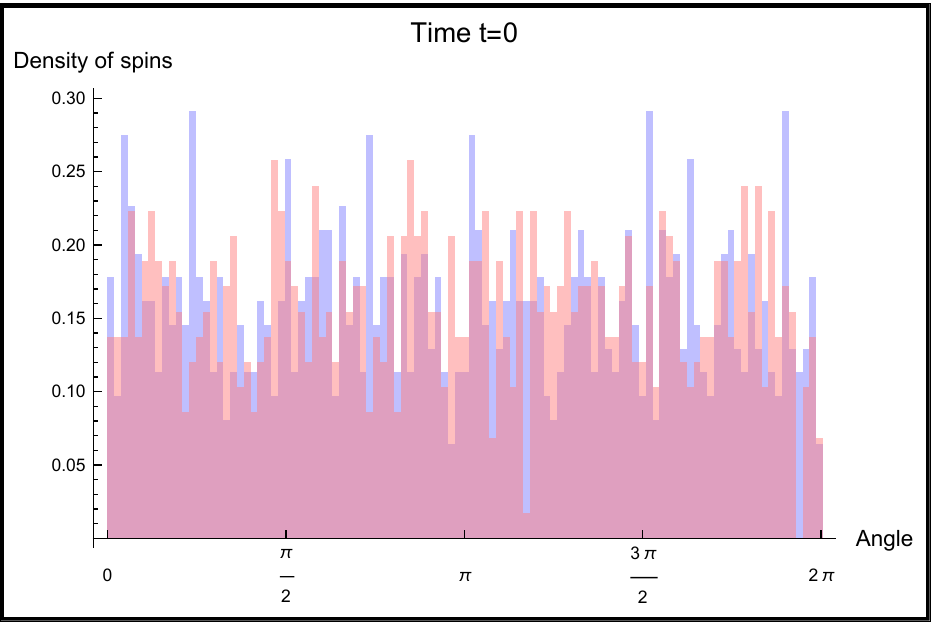}}
\subfigure{\includegraphics[width=.3\textwidth]{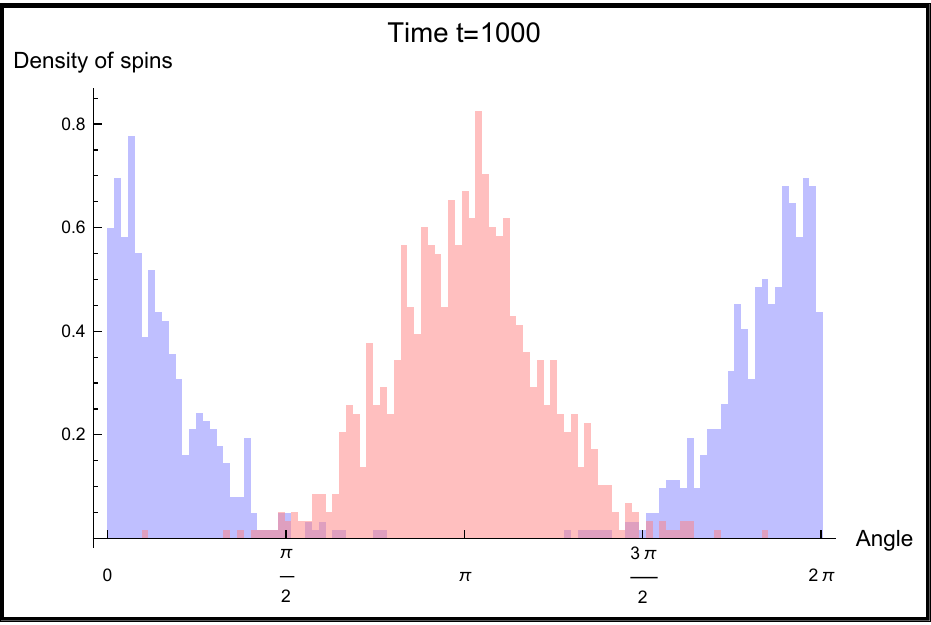}}
\subfigure{\includegraphics[width=.3\textwidth]{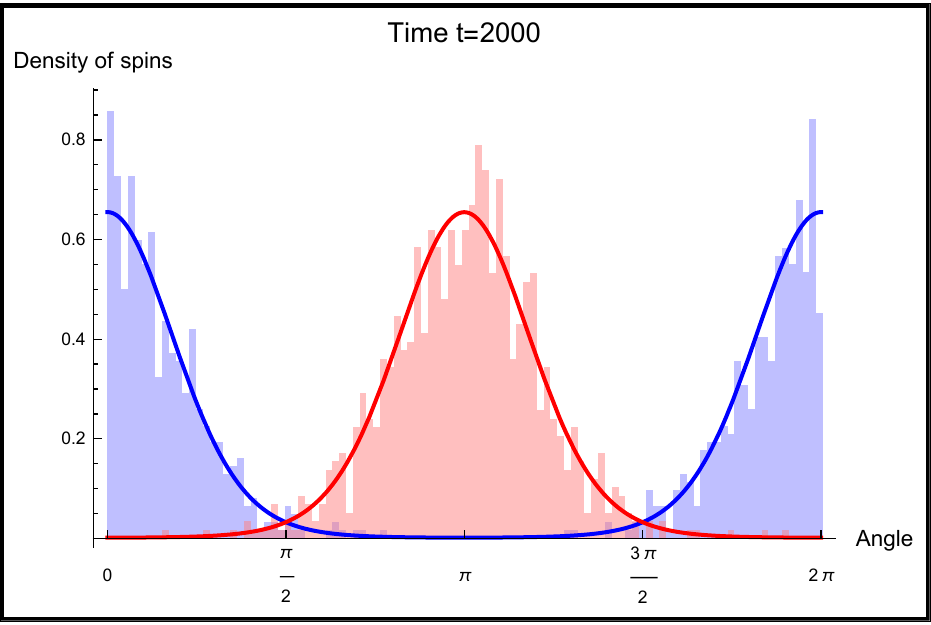}}
\caption{The picture shows the time-evolution of interacting particle system \eqref{InfGenXY:Micro} with $N=2000$ spins in the case when $h=1.5$ and $\theta=0.5$ (phase \textsf{0} in Fig.~\ref{fig:ph:dia}). Simulations have been run for $2000$ iterations, with time-step $dt=0.01$, and starting from an initial uniform random configuration on $[0,2\pi)^{N}$. The three snapshot histograms correspond to the configurations of the system at times $t=0$, $1000$ and $2000$ respectively. In each panel, on the $x$-axis we have the interval $[0,2\pi)$ and on the $y$-axis the normalized number of spins lying on the same angular position is registered; in blue (resp. red) spins subject to positive (resp. negative) field are displayed. We can see that in the long-run the spins tend to align with the site-dependent magnetic fields. As a result, approximately half of the particles are around angle $0$ and the other half around $\pi$, giving a null total magnetization; more precisely, we get $r_N (t_{fin}) = 0.001$. In the last snapshot the limiting distributions for the two families of spins are superposed; the solid blue (resp. red) line is $q(x,+1)$ (resp. $q(x,-1)$) defined by \eqref{StatSol:MKV}.
}
\label{fig:hist:1}
\end{figure}

But when $\theta$ exceeds $\theta_1$, this paramagnetic state seems to lose stability and $r_N(t)$ grows, reflecting the formation of a small cluster of spins that are aligned, thereby generating a collective phenomenon. Eventually $r_N(t)$ saturates at some level $0 < r_N(\infty) < 1$ (see Fig.~\ref{fig:hist:2}). With further increases in $\theta$, more and more spins are recruited into the ``aligned cluster'' and $r_N(\infty)$ grows as shown in Fig.~\ref{fig:theta:r}.

\begin{figure}[h!]
\centering
\subfigure[]{
\includegraphics[width=.3\textwidth]{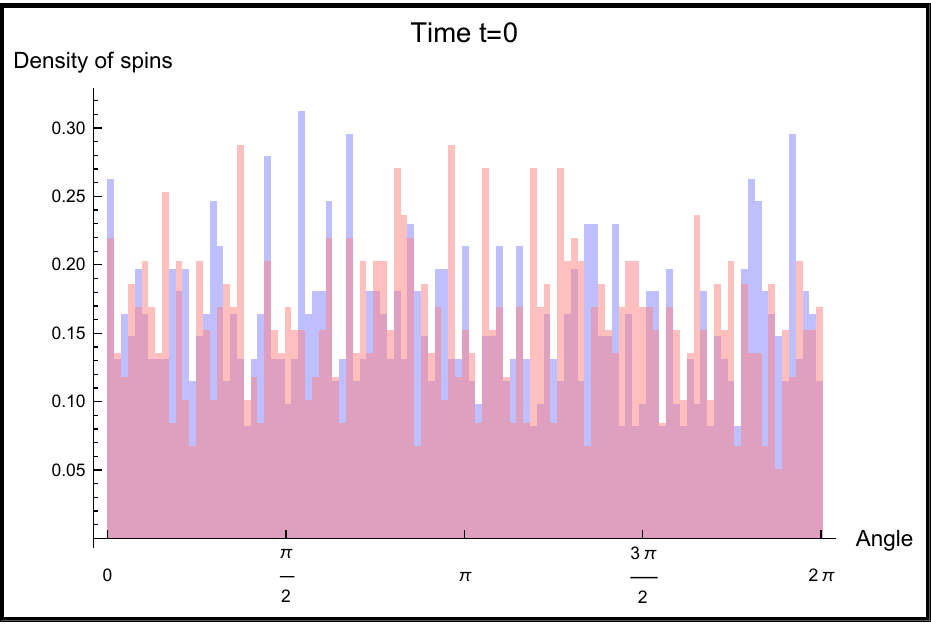}
\includegraphics[width=.3\textwidth]{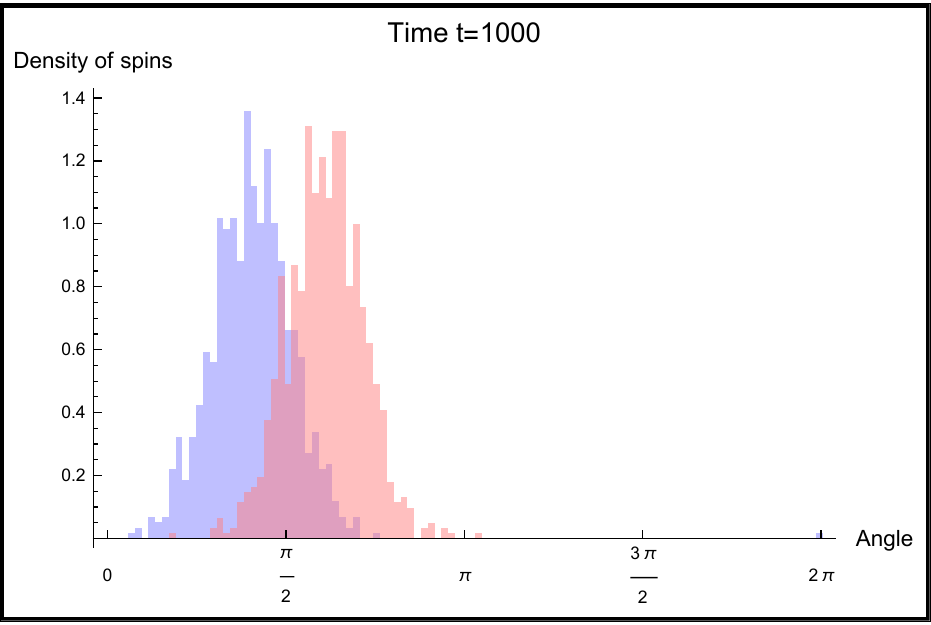}
\includegraphics[width=.3\textwidth]{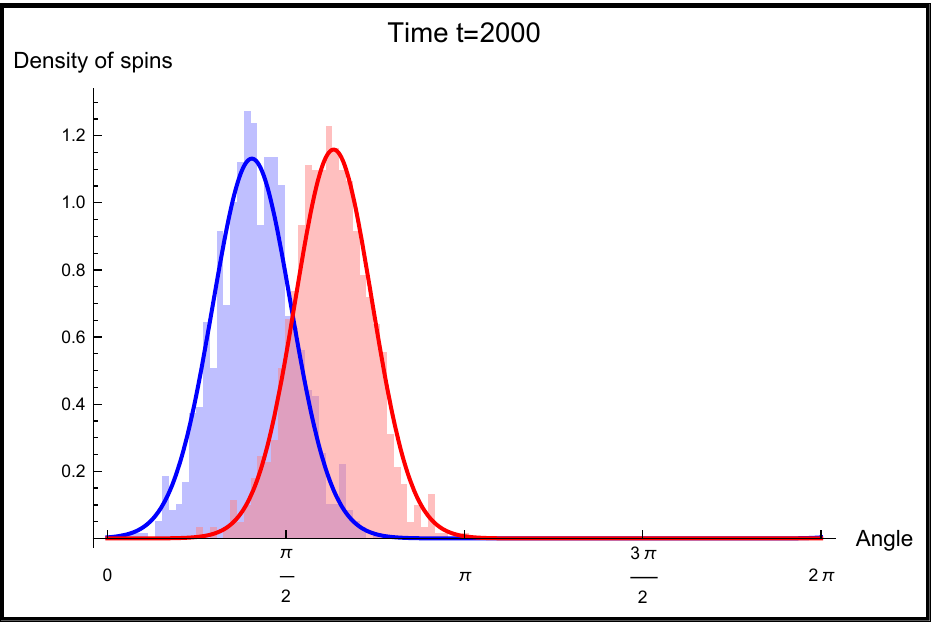}
\label{fig:sup:ev}
}
\subfigure[]{
\includegraphics[width=.5\textwidth,height=3cm]{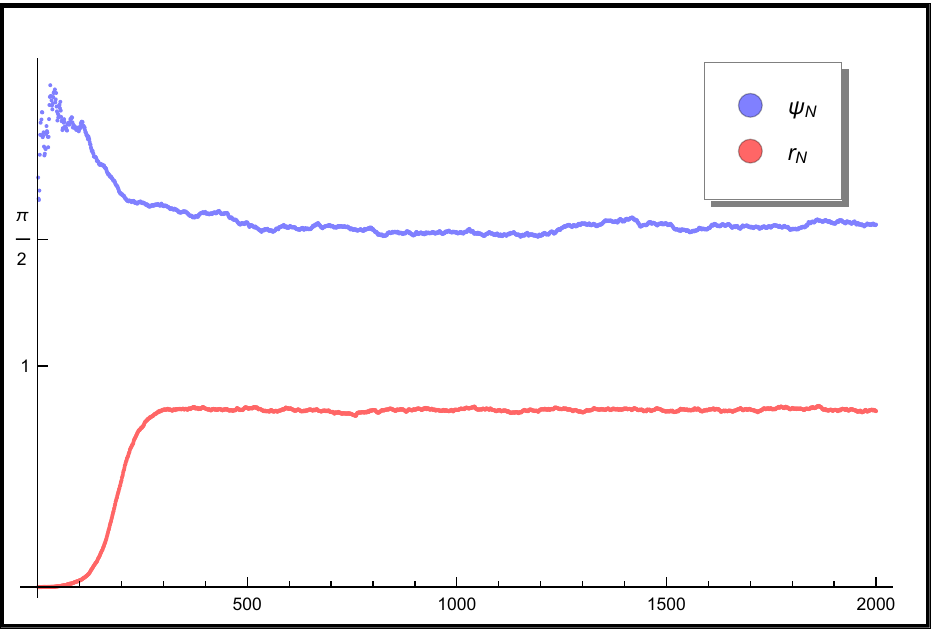}
\label{fig:r:Psi}
}
\caption{The picture shows the time-evolution of interacting particle system \eqref{InfGenXY:Micro} with $N=2000$ spins in the case when $h=1.5$ and $\theta=5$ (phase \textsf{6} in Fig.~\ref{fig:ph:dia}). Simulations have been run for $2000$ iterations, with time-step \mbox{$dt=0.01$}, and starting from an initial uniform random configuration on $[0,2\pi)^{N}$.  (a) The three snapshot histograms correspond to the configurations of the system at times $t=0$, $1000$ and $2000$ respectively.  In each panel, on the $x$-axis we have the interval $[0,2\pi)$ and on the $y$-axis the normalized number of spins lying on the same angular position is registered; in blue (resp. red) spins subject to positive (resp. negative) field are displayed. We can see that in the long-run the spins point angles close to $\frac{\pi}{2}$. As a result, a spontaneous net magnetization appears; indeed, we have $r_N (t_{fin}) = 0.8$ and the mean orientation of the particles is $\Psi_N (t_{fin}) = 1.64$. In the last snapshot the limiting distributions for the two families of spins are superposed; the solid blue (resp. red) line is $q(x,+1)$ (resp. $q(x,-1)$) defined by \eqref{StatSol:MKV}. (b) The time-evolutions of $r_N$ and $\Psi_N$ corresponding to the current simulation are displayed.
}
\label{fig:hist:2}
\end{figure}


\begin{figure}[h!]
\centering
\includegraphics[width=.45\textwidth]{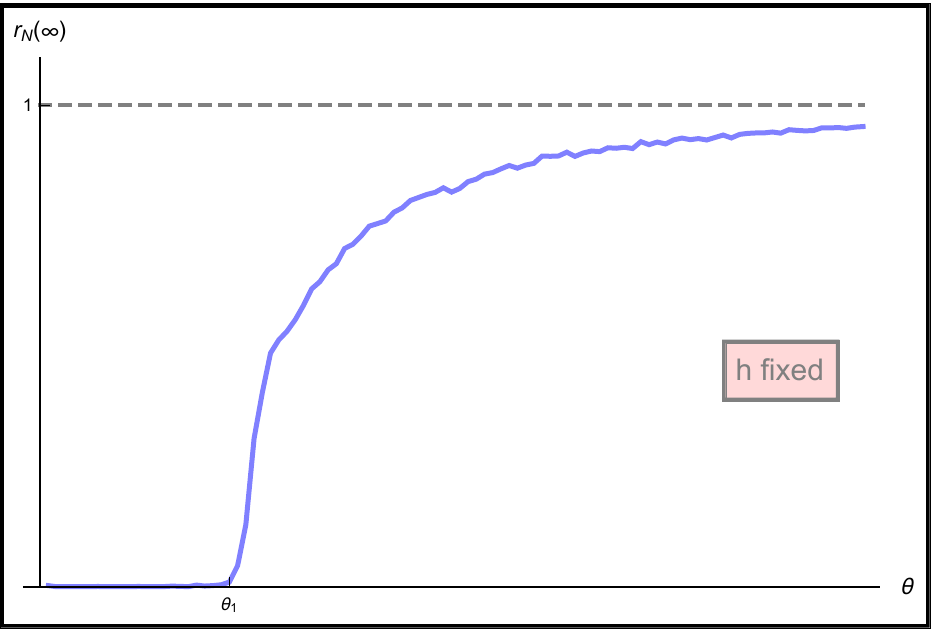}
\caption{Dependence of the steady-state coherence $r_N(\infty)$ on the coupling strength $\theta$ for fixed~$h$.}
\label{fig:theta:r}
\end{figure}

Simulations indicate that the paramagnetic solution is globally stable for $\theta < \theta_1$; whereas, it becomes unstable for $\theta > \theta_1$, when ferromagnetic equilibria arise. In the multiplicity phase the numerics further suggest that the dynamics approach either $\left( r_+, \frac{\pi}{2} \right)$ or $\left( r_+, \frac{3\pi}{2} \right)$. All the other ferromagnetic stationary points are unstable. In other words, it seems there are only two attracting states for each value of $\theta > \theta_1$. Indeed dynamical simulations are confirmed by the analysis of the free energy. The model is reversible and therefore the evolution is driven by a free energy $\mathcal{F}_{\theta,h}$ that corresponds (up to an additive constant) to the large deviation functional of the invariant measure. Let $q$ be a stationary solution of \eqref{MKV:Equations}; following \cite{DaGa89} we obtain
\begin{multline}\label{free:energy}
\mathcal{F}_{\theta,h} (r_q, \Psi_q) = \int_{\{-1,1\}}\int_{0}^{2\pi} \left[ -\frac{\theta}{2} \, r_q \cos(\Psi_q-x) - h\eta \cos x \right] q(x,\eta) \, dx  \, \mu(d\eta) \\
+ \frac{1}{2} \int_{\{-1,1\}} \int_{0}^{2\pi} q(x,\eta)\log[2\pi q(x,\eta)] \, dx \, \mu(d\eta).
\end{multline}
To study the stability of the various equilibria we found, we checked their relative heights on the free energy surface: we solved numerically the self-consistency relation \eqref{eqBessel} and we plugged the obtain pair(s) $(r,\Psi)$ in \eqref{free:energy}. Moreover, to visualize the energy landscape, we plotted directly the surface \eqref{free:energy}. See Fig.~\ref{fig:free:energy} for an example. We tested several choices of $\theta$ and $h$ for each region of the parameter space.  In all multiplicity phases  the free energy functional has minima at the spin-flop points $\left( r_+, \frac{\pi}{2} \right)$ and $\left( r_+, \frac{3\pi}{2} \right)$. Whereas, in phases {\sf 4} and {\sf 6}, the ferromagnetic solutions having $\Psi \in \{0,\pi\}$ are either maxima or saddles for $\mathcal{F}_{\theta, h}$ and hence always unstable. This study  supports the idea that a \emph{spin-flop transition} occurs when increasing the coupling strength.\\ 
Finally we would remark that, if we fix $h$ and vary $\theta$, there is no ordering between $r_+(\theta)$, $\overline{r}_+(\theta)$ and $\hat{r}_+(\theta)$. On the other hand, for fixed $\theta$ and variable $h$, we have a monotone relation between the different coherence indicators seen as functions of $h$. Namely, in phase {\sf 4} we always get $r_+(h) \leq \overline{r}_+(h)$ and in regime {\sf 6} it holds $\hat{r}_+(h) \leq r_+(h) \leq \overline{r}_+(h)$.

\begin{figure}[p]
\centering
\subfigure[Phase \textsf{0}.]{
\includegraphics[height=5cm]{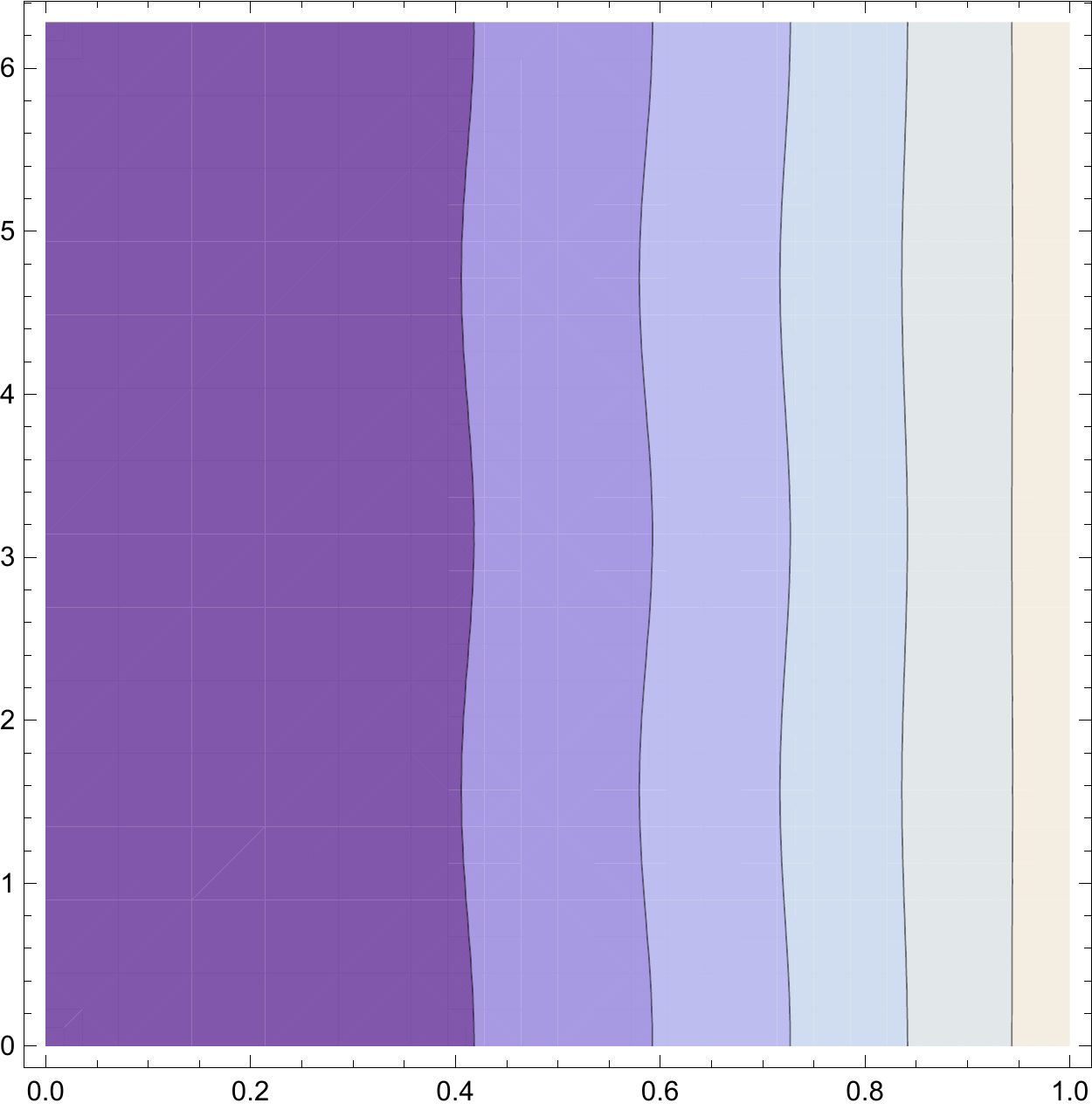}
\includegraphics[width=.45\textwidth]{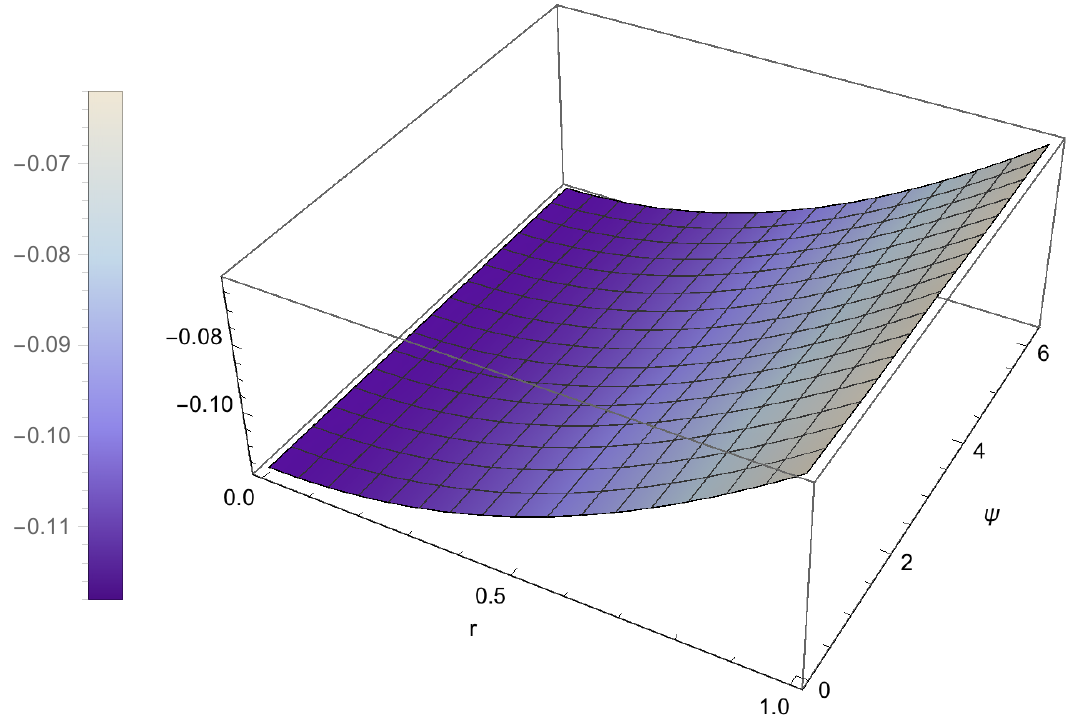}
}
\subfigure[Phase \textsf{2}.]{
\includegraphics[height=5cm]{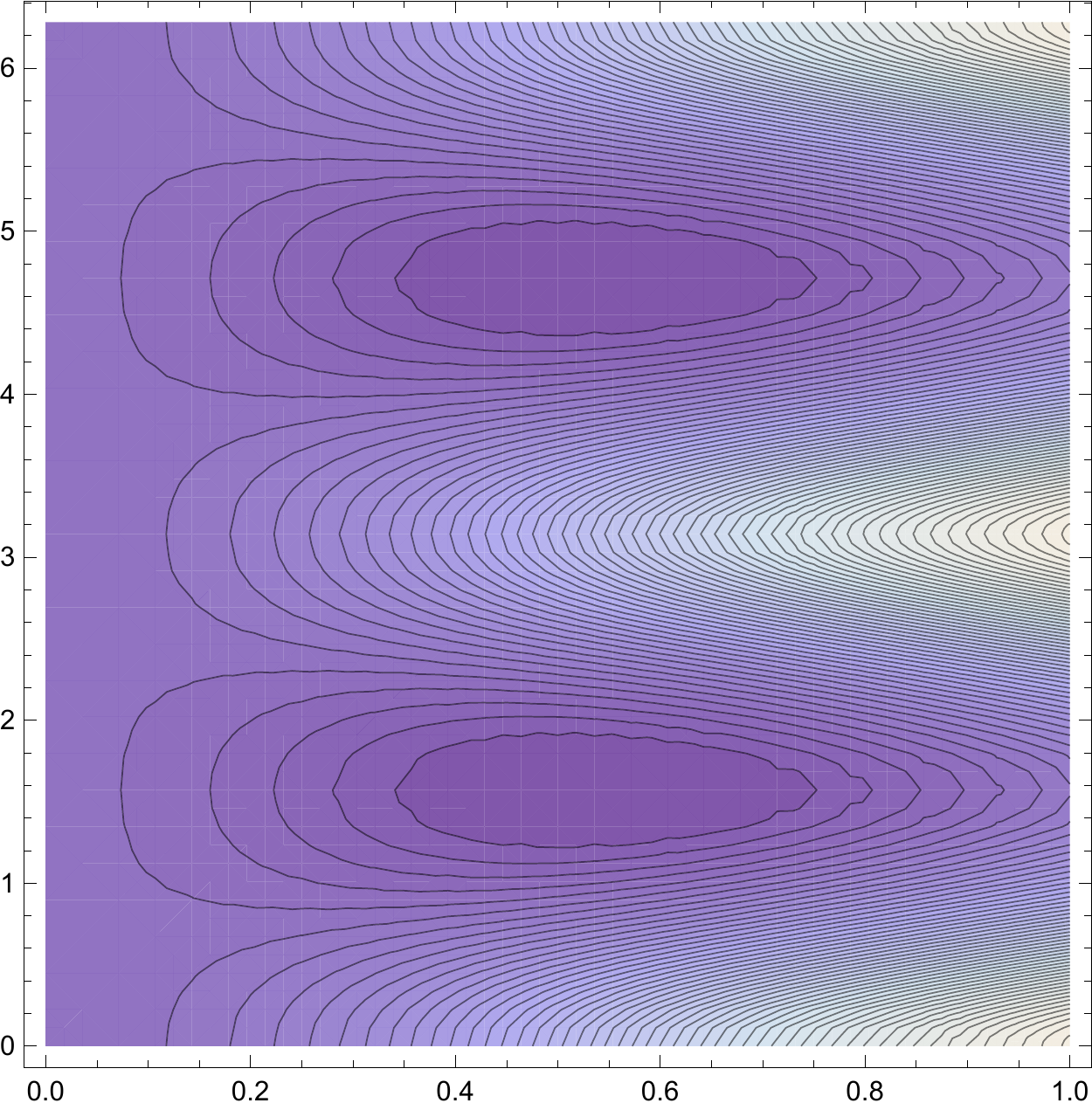}
\includegraphics[width=.45\textwidth]{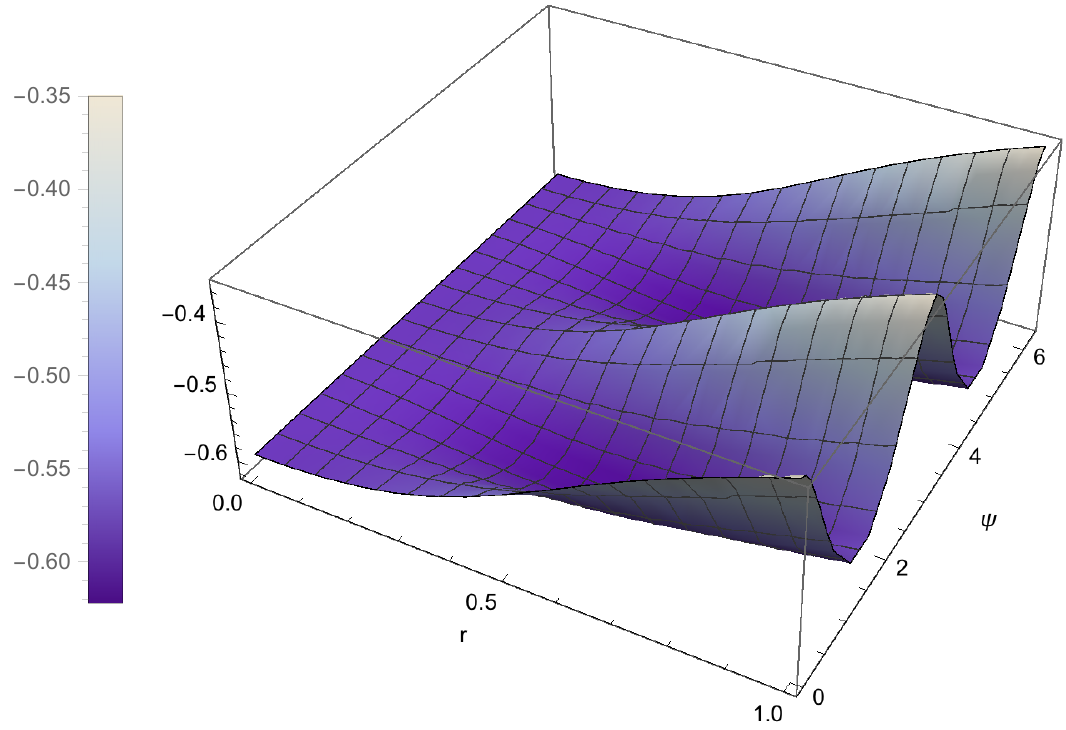}
}
\subfigure[Phase \textsf{4}.]{
\includegraphics[height=5cm]{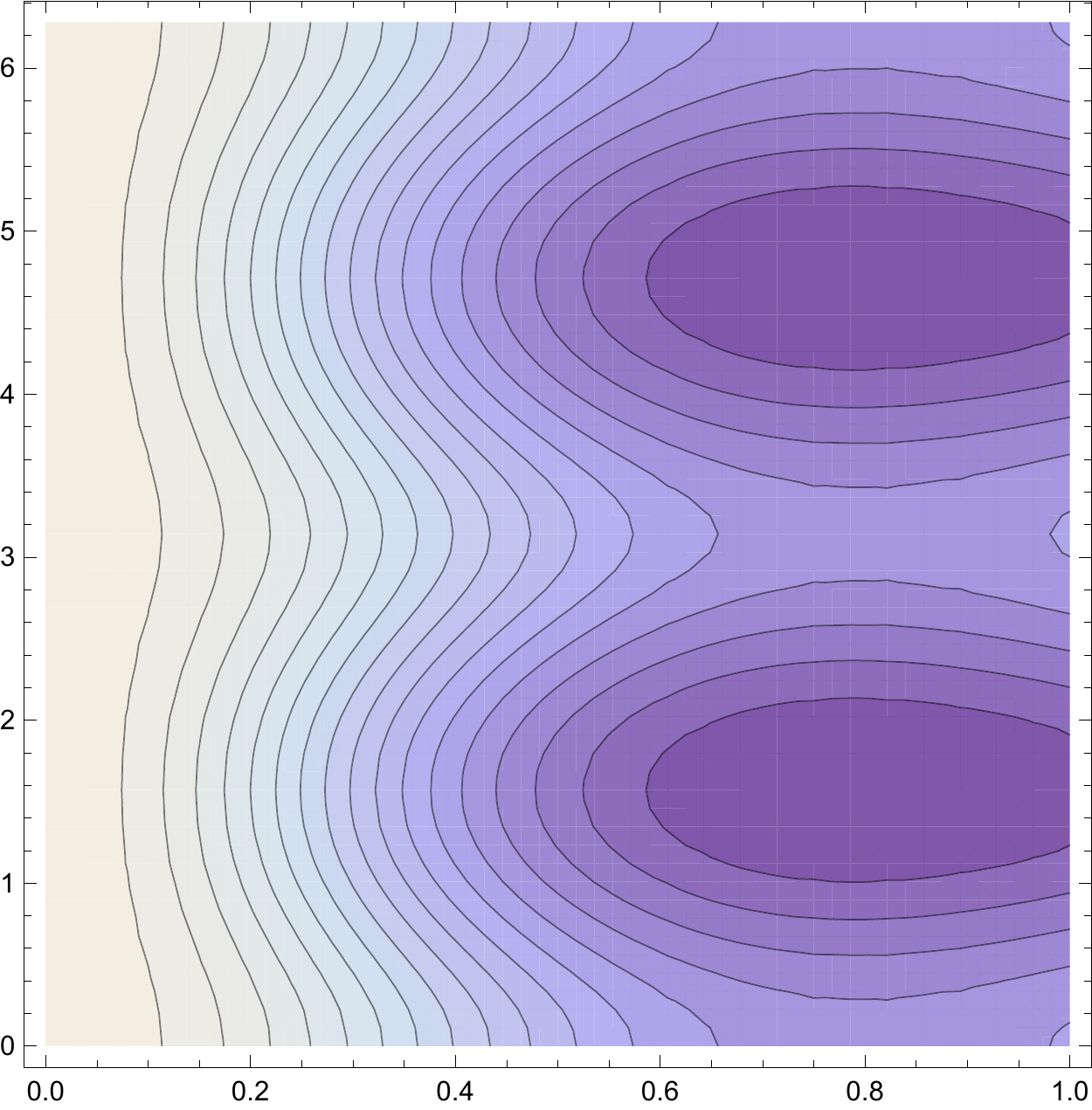}
\includegraphics[width=.45\textwidth]{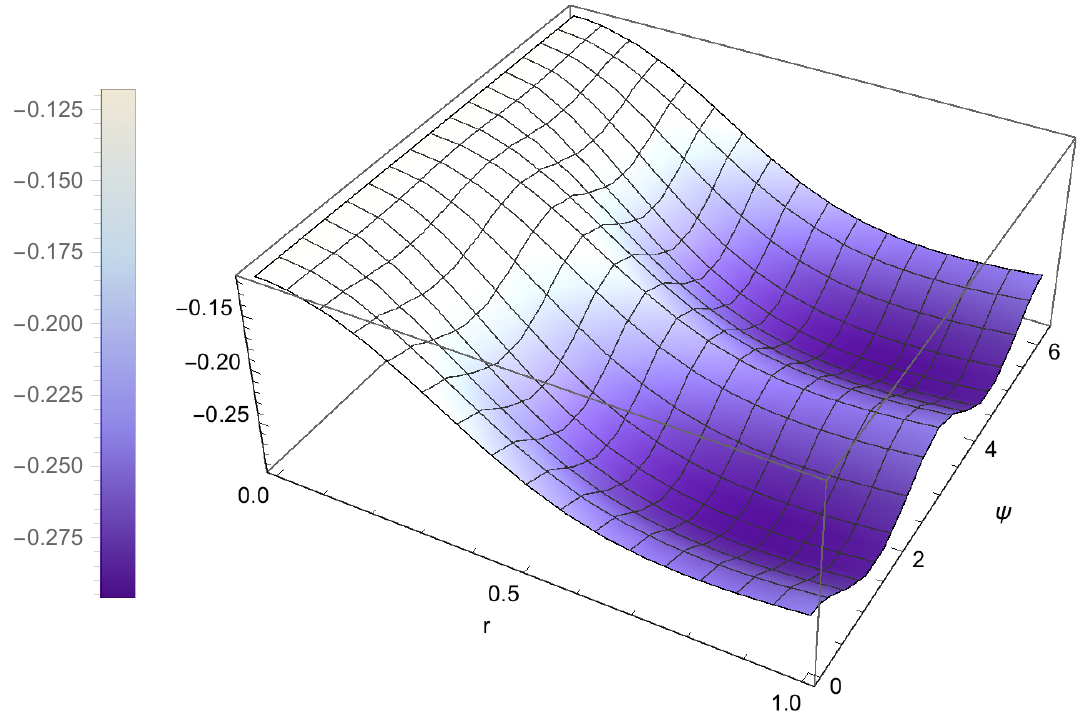}
}
\subfigure[Phase \textsf{6}.]{
\includegraphics[height=5cm]{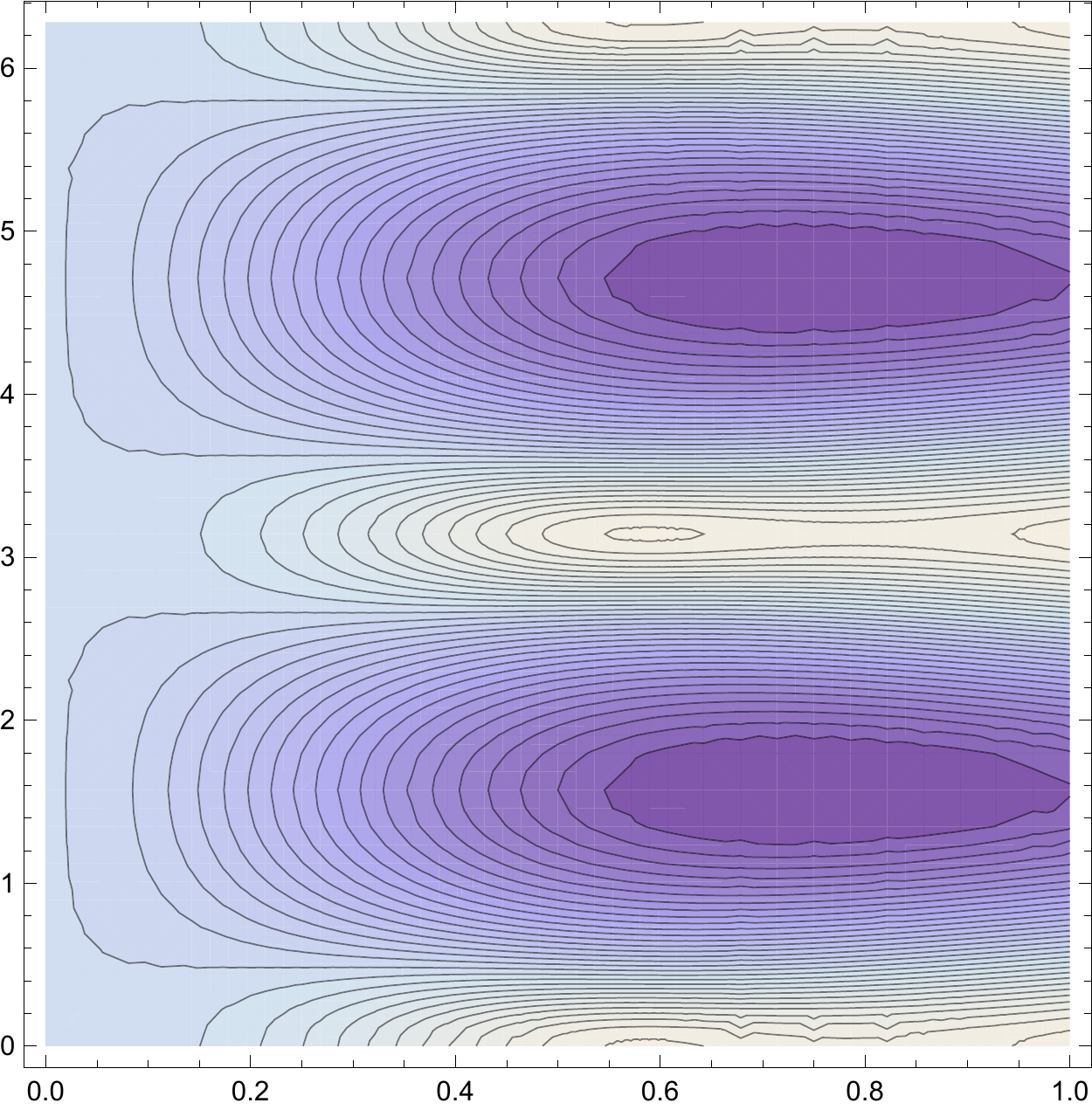}
\includegraphics[width=.45\textwidth]{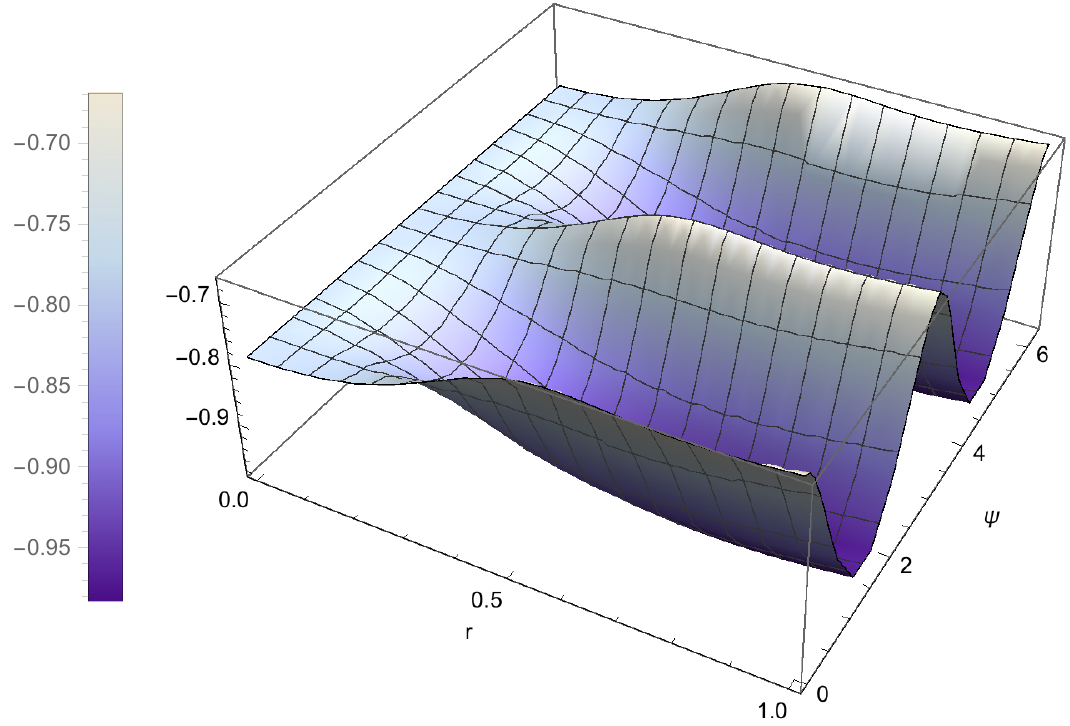}
}
\caption{The picture shows the free energy surface and the corresponding contour plot for several values of the parameters. A representative image for each of the phases in Fig.~\ref{fig:ph:dia} is displayed. Color convention: the darker the color, the lower the height of the surface.}
\label{fig:free:energy}
\end{figure}

\subsection{Critical fluctuations}

An important and interesting further step would be to understand how macroscopic observables fluctuate around their mean values when the system is put at the critical point. In this regime to obtain a limit theorem describing the fluctuations of the empirical measure process as $N \to +\infty$ we construct a process of the form 
\begin{equation}\label{fluct}
 N^{\frac{1}{4}} \left[ \rho_N \left( N^{\alpha}t \right) - q \right] 
\end{equation}
for suitable $\alpha > 0$. There are two notable features of this rescaling. On the one hand, we have a non-Gaussian spatial scale ($N^{\frac{1}{4}}$ instead of $N^{\frac{1}{2}}$); this implies that critical fluctuations are spatially larger than non-critical ones. On the other, the process must be observed in fast time $N^{\alpha}t$ because of the phenomenon of ``critical slowing down''. It means that the fluctuations persist over long time scale.\\ 
We would like to determine the exponent $\alpha$ such that \eqref{fluct} admits a meaningful limit in the sense of weak convergence. It is not clear a priori what to expect as time scale. The addition of disorder may have a drastic impact on the fluctuation process and change the time scale at which it exists. \\
In the paper \cite{CoDaP12} the authors analyze how disorder affects the dynamics of critical fluctuations for two different types of interacting particle system: the Curie-Weiss and Kuramoto models in random environment. The interesting point is that when disorder is added, spin and rotator systems belong to different universality classes, which is not the case for their non-disordered counterparts. Hence the disorder is responsible for destroying universality. Roughly speaking in the Curie-Weiss model the fluctuations produced by the disorder always prevail in the critical regime: these fluctuations evolve in a time scale which is much shorter ($\alpha=\frac{1}{4}$) than the corresponding one for homogeneous system ($\alpha=\frac{1}{2}$). For rotators, the disorder does not modify the ``standard'' slowing down ($\alpha=\frac{1}{2}$).\\
The question is why does it happen? The random Kuramoto model in \cite{CoDaP12} is not reversible\footnote{
\textbf{Kuramoto model in random environment \cite{CoDaP12}.} Given a configuration \mbox{$\underline{x} \in [0,2\pi)^N$} and a realization of the random environment $\underline{\eta} \in \{ -1,+1\}^N$, we can define the Hamiltonian $H_N(\underline{x},\underline{\eta}):  [0,2\pi)^N \times \{ -1,+1\}^N  \longrightarrow  \mathbb{R}$ as
\begin{equation}\label{Hamiltonian:K}
 H_N(\underline{x},\underline{\eta})=-\frac{\theta}{2N}\sum_{j,k=1}^N \cos(x_k-x_j) - h \sum_{j=1}^N  \eta_j  x_j \,,
\end{equation}
where $x_j$ is the position of rotator at site $j$; the disorder term $h \eta_j$, with $h>0$, can be interpreted as its own frequency and $\theta > 0$ is the coupling strength. It is important to notice that the system \eqref{Hamiltonian:K} is not reversible unless $h=0$.
}
 and moreover presents a discrepancy between the symmetry type of the state and the disorder variables (rotational vs. up/down). We wonder if this difference is due to the reversibility/irreversibility or rather symmetry issues. The general idea is therefore to consider two modifications of the random Kuramoto model, aimed at getting a reversible system with disorder having either up/down or rotational symmetry, and then make a comparison between the time scale of critical fluctuations. The XY model can be read as the variation that accounts for the reversibility plus up/down symmetry case. As a first step it would be interesting to investigate its critical fluctuations and compare them to those of the Curie-Weiss model. 
\\ 

Both stability properties of the steady solutions and the behavior of fluctuations appear to be difficult to determine analytically and are left unsolved by the present paper. To prove them rigorously one should have the complete control over the spectrum of the linearization of the operator \eqref{Operator:MKV} that is an open and difficult problem at the moment. We feel that this analysis may deserve much more room and defer some more detailed work to future research. 

\section{Proofs}\label{sc:proofs}

\subsection{Proof of Proposition \ref{prop:stat:sols}}

An equilibrium probability density for \eqref{MKV:Equations} must satisfy
\begin{equation}\label{1}
\frac{1}{2}\frac{\partial^2 q}{\partial x^2}(x,\eta) = \frac{\partial}{\partial x} \left\{ \left[ \theta r \sin(\Psi - x) - h \eta \sin x \right] q(x,\eta) \right\}
\end{equation}
Respecting normalization ($\int_0^{2\pi} q(x,\eta) \, dx = 1$) and periodic boundary conditions ($q(0,\eta)=q(2\pi,\eta)$ for every $\eta$), we can solve \eqref{1} and complete the proof.

%

\subsection{Proofs of Proposition \ref{prop:r=0} and Proposition \ref{integralr}}

Every stationary solution $(r,\Psi)$ has to satisfy the self-consistency relation \eqref{Order:Parameter:Stat}, which is equivalent to conditions 
\begin{align}
r  &= \int_{\{-1,+1\}} \int_0^{2\pi} \cos( x-\Psi ) \, q (x,\eta) \, dx \, \mu(d\eta) \label{eqRe1}
\\
0 &= \int_{\{-1,+1\}} \int_0^{2\pi} \sin(x-\Psi ) \,  q (x,\eta) \, dx \, \mu(d\eta) \,, \label{eqIm1} 
\end{align}
where $q (x,\eta)=[Z(\eta)]^{-1} \cdot \exp \{2\theta r \cos(\Psi -x)+2h\eta\cos x \}$. By standard trigonometric formulas, equations \eqref{eqRe1} and \eqref{eqIm1} can be rewritten as 
%
%
%
%
\begin{align}
r  & = \frac{1}{2} \cos \Psi\biggl [ \int_0^{2\pi} \cos x \, q (x,+1)dx + \int_0^{2\pi} \cos x \, q (x,-1)dx\biggr] \nonumber\\
& \qquad \qquad + \frac{1}{2} \sin \Psi\biggl [\int_0^{2\pi}\sin x \, q (x,+1) dx +\int_0^{2\pi}\sin x \, q (x,-1) dx\biggr] \label{eqRe} \tag{\ref{eqRe1}$'$} \\
& \nonumber \\
0 & =  \cos \Psi \biggl [\int_0^{2\pi} \sin x \, q (x,+1) dx +\int_0^{2\pi}\sin x \, q (x,-1) dx\biggr ] \nonumber\\
& \qquad \qquad -  \sin \Psi\biggl[ \int_0^{2\pi} \cos x \, q (x,+1)dx + \int_0^{2\pi} \cos x \, q (x,-1)dx\biggr ]. \label{eqIm} \tag{\ref{eqIm1}$'$}
\end{align}
All the integrals involved can be rephrased in terms of Bessel functions. We make the main steps explicit for $\int_0^{2\pi} \cos x \, q (x, +1) \, dx$, the remaining integrals can be dealt with similarly. We have,%
\begin{align}
Z(+1) \int_0^{2\pi} \cos x \, q (x,+1) dx 
&=\sum_{k \in 2\mathbb{N}+1} \frac{1}{k!} \int_0^{2\pi} \cos x \, \Bigl [2\theta r \cos(\Psi -x)+2h\cos x \, \Bigr ]^k dx \nonumber\\
&= 2 \pi (h+\theta r \cos \Psi) \sum_{k=0}^{+\infty} \frac{\left( h^2+\theta^2r ^2+2\theta h r \cos \Psi\right)^k}{\Gamma(k+2)\Gamma(k+1)} \nonumber\\
& =2\pi (h+\theta r \cos \Psi)\frac{I_1\left( 2\sqrt{h^2+\theta^2r ^2+2\theta h r \cos \Psi} \right)}{\sqrt{h^2+\theta^2r ^2+2\theta h r \cos \Psi}}, \label{conv:Bessel}
\end{align}
where $\Gamma ( \cdot)$ is the Gamma function and $I_{v} (\cdot)$ denotes the first kind modified Bessel function of order~$v$. The derivation of \eqref{conv:Bessel} is postponed to Appendix~\ref{app:integral}. In the same manner we calculate
%
%
%
%
%
\[
Z(-1) \int_0^{2\pi} \cos x \, q (x,-1) dx = - 2\pi (h-\theta r \cos \Psi) \, \frac{I_1\left( 2\sqrt{h^2+\theta^2r ^2-2\theta h r \cos \Psi} \right)}{\sqrt{h^2+\theta^2r ^2-2\theta h r \cos \Psi}},
\]
\[
Z(\pm 1) \int_0^{2\pi} \sin x \, q (x,\pm 1) dx = 2\pi (\theta r \sin \Psi) \, \frac{I_1\left( 2\sqrt{h^2+\theta^2r ^2 \pm 2\theta h r \cos \Psi} \right)}{\sqrt{h^2+\theta^2r ^2 \pm 2\theta h r \cos \Psi}}
\]
and the normalizing constants
\[
Z(\pm 1) = 2\pi I_0 \left( 2\sqrt{h^2+\theta^2r ^2 \pm 2\theta h r \cos \Psi} \right).
\]
By plugging what we obtained into equations \eqref{eqRe} and \eqref{eqIm}, we get
\begin{align}
r &= \frac{\theta r + h \cos \Psi}{\sqrt{h^2+\theta^2r ^2+2\theta h r \cos \Psi}} \, \frac{I_1 \left( 2\sqrt{h^2+\theta^2r ^2+2\theta h r \cos \Psi} \right)}{I_0 \left( 2\sqrt{h^2+\theta^2r ^2+2\theta h r \cos \Psi} \right)} \nonumber \\
& \qquad \qquad  \qquad + \frac{\theta r - h \cos \Psi}{\sqrt{h^2+\theta^2r ^2 - 2\theta h r \cos \Psi}} \, \frac{I_1 \left( 2\sqrt{h^2+\theta^2r ^2 - 2\theta h r \cos \Psi} \right)}{I_0 \left( 2\sqrt{h^2+\theta^2r ^2 - 2\theta h r \cos \Psi} \right)}  \label{eqRe3} \tag{\ref{eqRe1}$''$}\\
& \nonumber \\
0 &= h \sin \Psi \left[ \frac{I_1 \left( 2\sqrt{h^2+\theta^2r ^2+2\theta h r \cos \Psi} \right)}{\sqrt{h^2+\theta^2r ^2+2\theta h r \cos \Psi} \, I_0 \left( 2\sqrt{h^2+\theta^2r ^2+2\theta h r \cos \Psi} \right)} \right. \nonumber\\
&\qquad \qquad \qquad  \left. - \frac{I_1 \left( 2\sqrt{h^2+\theta^2r ^2-2\theta h r \cos \Psi} \right)}{ \sqrt{h^2+\theta^2r ^2-2\theta h r \cos \Psi} \, I_0 \left( 2\sqrt{h^2+\theta^2r ^2-2\theta h r \cos \Psi} \right)}\right].\label{eqIm3} \tag{\ref{eqIm1}$''$} 
\end{align}
Focus on equation \eqref{eqIm3}. It is equivalent either to $\sin \Psi = 0$ or the term into square brackets vanishes. By Lemma~\ref{tech:lemma} in Appendix~\ref{app:Bessel} we know that the function $g(z)=\frac{I_1(2\sqrt{z})}{\sqrt{z} \, I_0(2\sqrt{z})}$ is strictly decreasing on $]0,+\infty[$ and so the latter circumstance occurs if and only if the arguments of the two functions appearing into brackets are equal. Therefore, equation \eqref{eqIm3} admits solution when: (a) $\Psi = 0$ or $\Psi = \pi$; (b) $\Psi = \frac{\pi}{2}$ or $\Psi = \frac{3\pi}{2}$; (c) $r=0$. \\

It is easy to see that under (c) equation \eqref{eqRe3} is always satisfied, giving the statement of Proposition~\ref{prop:r=0}. Moreover, under (a) or (b) equation \eqref{eqRe3} reduces to \eqref{eqBessel} and this concludes the proof of Proposition~\ref{integralr}.

\subsection{Proof of Theorem \ref{thm:ph:dia}}

From Proposition~\ref{integralr} we infer that we have only to consider the cases when $\Psi \in \left\{ 0, \frac{\pi}{2}, \pi, \frac{3\pi}{2} \right\}$.  We divide the study in a few steps.\\

\textbf{Analysis for} $\Psi \in \left\{ \frac{\pi}{2}, \frac{3\pi}{2} \right\}$. We stick on the case $\Psi = \frac{\pi}{2}$, the other being similar. First we prove that for $\theta \leq \theta_1 (h)$ there are only paramagnetic solutions; whereas, for $\theta > \theta_1 (h)$, there exists a \emph{unique} ferromagnetic solution.\\
If we set $\Psi=\frac{\pi}{2}$, then $q(x,\eta) = [ Z(\eta)]^{-1} \cdot \exp \left\{ 2 \theta r \sin x + 2 h \eta \cos x\right\}$ and the self-consistency relation \eqref{Order:Parameter:Stat} is equivalent to the conditions 
\begin{align}
r &= \int_{\{-1,+1\}} \int_0^{2\pi} \sin x \, q(x,\eta) \, dx \, \mu(d\eta) \label{r}
\\
0 &= \int_{\{-1,+1\}} \int_0^{2\pi} \cos x \,  q(x,\eta) \, dx \, \mu(d\eta) \,.\label{0}
\end{align}
We must show that \eqref{r} has a positive solution and that \eqref{0} is always satisfied.\\
First observe that
\begin{align*}
\int_{\{-1,+1\}} \int_0^{2\pi} \sin x \, q(x,\eta) \, dx \, \mu(d\eta) &= \frac{1}{2} \left[ \int_0^{2\pi} \sin x \, q(x,-1) \, dx \,  + \int_0^{2\pi} \sin x \, q(x,+1) \, dx \,  \right]\\
&\hspace{-0.5cm} \stackrel{\mbox{\tiny ($y=\pi-x$)}}{=} \frac{1}{2} \left[ \int_{-\pi}^{\pi} \sin y \, q(y,+1) \, dy + \int_{0}^{2\pi} \sin x \, q(x,+1) \, dx \right] \\
&\hspace{-0.6cm} \stackrel{\mbox{\tiny (periodicity)}}{=} \int_{0}^{2\pi} \sin x \, q(x,+1) \, dx
\end{align*}
and analogously
\[
\int_{\{-1,+1\}} \int_0^{2\pi} \cos x \, q(x,\eta) \, dx \, \mu(d\eta) = 0 \,.
\]
Therefore, \eqref{0} is proved and it remains to show that \eqref{r} admits a solution $r>0$. Let us define the functional
\[
F_1(r) := \frac{\displaystyle{\int_0^{2\pi} \sin x \, \exp \{2 \theta r \sin x + 2 h \cos x\} \, dx}}{\displaystyle{\int_0^{2\pi} \exp \{2 \theta r \sin x + 2 h \cos x\} \, dx}}  \,.
\]
We look for a positive solution of the fixed point equation $r = F_1(r)$. Observe that by \eqref{eqBessel} we can rewrite $F_1$ as
\[
F_1(r) = \frac{\theta r}{\sqrt{h^2+\theta^2r^2}} \frac{I_1(2\sqrt{h^2+\theta^2r^2})}{I_0(2\sqrt{h^2+\theta^2r^2 })}
\]
and, since we are interested in solutions $r \neq 0$, our problem translates in finding a positive solution for the equation
\[
\tilde{F}_1(r) = 1,
\]
where $\tilde{F}_1(r) := \frac{\theta}{\sqrt{h^2+\theta^2r^2}} \frac{I_1(2\sqrt{h^2+\theta^2r^2})}{I_0(2\sqrt{h^2+\theta^2r^2 })}$. We have

\begin{itemize}
\item[$\bullet$] $\tilde{F}_1(0) = \frac{\theta}{h} \frac{I_1(2h)}{I_0(2h)}$ and $\tilde{F}_1$ is continuous in $[0,1]$ for all the values of the parameters.
\item[$\bullet$] $\tilde{F}_1(r)$ is strictly decreasing in $]0,1]$. Note that $\tilde{F}_1(r) = (g \circ f) (r)$ with $g(z) := \frac{I_1(2\sqrt{z})}{\sqrt{z} \, I_0(2\sqrt{z})}$ and $f(z) := h^2 + \theta^2 z^2$. Moreover, $g$ is strictly decreasing on $]0,1]$ by Lemma~\ref{tech:lemma} in Appendix~\ref{app:Bessel}; whereas, $f$ is strictly increasing. Therefore $F_1$ is strictly decreasing as it is a composition of a decreasing and an increasing function.
\end{itemize}

Hence, if $\tilde{F}_1(0) > 1$ then $\tilde{F}_1(r)$ intersects the horizontal line $1$ exactly once; on the contrary, whenever $\tilde{F}_1(0) \leq 1$ there are no crosses. To conclude it is sufficient to notice that $\theta_1(h)$ as defined in the statement of Theorem~\ref{thm:ph:dia} equals $h \frac{I_0(2h)}{I_1(2h)}$.

\bigskip

\textbf{Analysis for} $\Psi \in \left\{ 0, \pi \right\}$. We stick on the case $\Psi = 0$, the other being similar. We want to show that for $\theta > \theta_2 (h)$, there is \emph{exactly one} ferromagnetic solution and moreover, that there exists a further critical value $\theta_\star (h)$, with $\theta_\star (h) < \theta_2 (h)$, such that if \mbox{$\theta \leq \theta_\star (h)$} there are only paramagnetic solutions, while if $ \theta_\star (h) < \theta < \theta_2 (h)$ there are \emph{two} ferromagnetic ones.\\
If we set $\Psi=0$, then $q(x,\eta) = [Z(\eta)]^{-1} \cdot \exp \left\{ 2 (\theta r + h \eta ) \cos x \right\}$ and the self-consistency relation \eqref{Order:Parameter:Stat} is equivalent to the conditions 
\begin{align}
r &= \int_{\{-1,+1\}} \int_0^{2\pi} \cos x \, q(x,\eta) \, dx \, \mu(d\eta) \label{Re:part:2}
\\
0 &= \int_{\{-1,+1\}} \int_0^{2\pi} \sin x \,  q(x,\eta) \, dx \, \mu(d\eta) \,. \label{Im:part:2} 
\end{align}
We must show that \eqref{Re:part:2} has a positive solution and that \eqref{Im:part:2} is always satisfied.\\
First observe the $x$-integral in \eqref{Im:part:2} has an explicit anti-derivative that being $2\pi$-periodic makes the whole integral vanish for all values of the parameters. Therefore, \eqref{Im:part:2} is proved and it remains to show that \eqref{Re:part:2} admits solutions $r>0$. Let us define the functional
\[
F_2(r) :=  \int_{\{-1,+1\}} \,  \frac{\displaystyle{ \int_0^{2\pi} \cos x \, \exp \{ 2 (\theta r + h \eta ) \cos x \} \, dx}}{\displaystyle{\int_0^{2\pi} \exp \{ 2 (\theta r + h \eta ) \cos x \} \, dx}}  \, \mu (d\eta) \,.
\]
We look for a solution of the fixed point equation $r = F_2(r)$. We have
\begin{itemize}
\item[$\bullet$] 
$F_2(0)=0$ and $F_2$ is continuous in $[0,1]$ for all values of the parameters.
\item[$\bullet$] 
$\lim_{r \to +\infty} F_2(r) = 1$; indeed, as $r \to +\infty$ the function $x \mapsto \exp \{2 (\theta r +  h \eta ) \cos x \}$ becomes sharply peaked around $x = 0$ and so does also $x \mapsto \cos x \exp \{2 (\theta r + h \eta) \cos x \}$. Consequently, $\lim_{r \to +\infty} F_2(r) = \int_{\{-1,+1\}} \mu (d\eta) = 1$.
\item[$\bullet$] 
$F_2$ is strictly increasing; indeed, the first derivative of $F_2$ with respect to $r$ is given by
\begin{align*}
 F_2'(r) &= 2 \theta \int_{\{-1,+1\}} \biggl[ \int_0^{2\pi} \cos^2 x \, q(x,\eta) \, dx - 2 \theta  \biggl (  \int_0^{2\pi} \cos x \, q(x,\eta) \, dx \biggr  )^2 \biggr] \, \mu (d\eta) \\ 
              &=2 \theta \, \mathbb{E}_{\mu} \left[ \mathbb{V}ar_{q (x,\eta)} \left( \cos X \right) \right]
\end{align*}
which is a strictly positive quantity, since expected value of a variance. Moreover, it is readily seen that
\begin{align*}
F_2'(0) &= \theta \left[ \mathbb{V}ar_{q^{(0)} (x,+1)} \left( \cos X \right) + \mathbb{V}ar_{q^{(0)} (x,-1)} \left( \cos X \right)\right]\\
&\hspace{-0.6cm} \stackrel{\left( \parbox{1.2cm}{\tiny \centering $y = \pi - x$ \\ and \\ periodicity} \right)}{=} 2 \theta \, \mathbb{V}ar_{q^{(0)} (x,+1)} \left( \cos X \right).
\end{align*}
\item[$\bullet$] 
The second derivative is equal to 
\begin{equation}\label{d2F2}
F_2''(r) =  4\theta^2 \mathbb{E}_{\mu} \left\{ \mathbb{E}_{q(x,\eta)} \left[ \cos X - \mathbb{E}_{q(x,\eta)}(\cos X) \right]^3 \right\}
\end{equation}
and changes sign depending on the parameters, so that is not possible to conclude by a standard concavity argument.
Nevertheless from numerics we see that there is at most one sign change (we checked the number of zeros for \eqref{d2F2} in the region $[0,10] \times [0,10]$ of the parameter space $(h, \theta)$ on a grid of mesh-size $0.1$ and in the region $[0,10] \times [10,30] \cup [10,30] \times [0,30]$ on a grid of mesh-size $0.5$). As a consequence, $F_2$ changes the curvature at most once\footnote{This is the only point in our rigorous proof where we used numerical assistance. }.

Therefore we can argue as follows. Since as $r \rightarrow +\infty$ the function $F_2(r)$ approaches $1$ from below, it must be concave for large $r$. Then,
\begin{itemize}
\item
if $\theta \leq \theta_2(h)$ and $F_2''(r) \leq 0$ in a right-neighborhood of $r=0$, $F_2(r)$ is strictly concave on $[0,1]$ for any values of the parameters and hence there is no intersection with the diagonal.
\item
if $\theta \leq \theta_2(h)$ and $F_2''(r) > 0$ in a right-neighborhood of $r=0$, $F_2(r)$ changes curvature either below or above the diagonal, giving rise to none or precisely two positive fix points. The boundary between these two regions is represented by the dashed green line $\theta_\star (h)$ in Fig.~\ref{fig:ph:dia}. It has been obtained numerically and corresponds to the choice of parameters where there exists $r > 0$ such that $F_2(r)=r$ and $F_2'(r)=1$. Note that the curves $\theta_2$ and $\theta_\star$ coincide for $h \in \left[ 0,\bar{h} \right]$ and then separate at $h=\bar{h}$. 
\item
if $\theta > \theta_2(h)$, no matter if either $F_2''(r) \leq 0$ or $F_2''(r) > 0$ in a right-neighborhood of $r=0$, the curve $F_2(r)$ crosses the diagonal at precisely one positive $r$.
\end{itemize}  
We are left to understand which is the curvature of $F_2$ around $r=0$. To infer some information we Taylor expand the function and, by means of the representation \eqref{eqBessel}, we obtain
\[
F_2 (r) =  \theta \left[ \frac{I_0^2(2h) + I_0(2h)I_2(2h) -2I_1^2(2h)}{I_0^2(2h)} \right] r + \theta^3 K(h) r^3 + O(r^4)
\]
with
\begin{multline*}
K(h) := \frac{-3I_0^4(2h) +I_0^3(2h) \left[ I_4(2h) - 8I_2(2h)  \right] + I_0^2(2h) \left[ 24 I_1^2(2h) -6I_2^2(2h) - 8I_1(2h) I_3(2h)\right]}{I_0^4(2h)}\\ 
+ \frac{48 I_0(2h) I_1^2(2h) I_2(2h) - 48 I_1^4(2h)}{I_0^4(2h)},
\end{multline*}
where $I_v(\cdot)$ denotes a first kind modified Bessel function of order $v$. The function $K(h)$  admits a (unique) zero at $\bar{h} \simeq 0.514443$. The graph of $K(h)$ is shown in Fig.~\ref{fig:kappa}. Therefore, if $h \leq \bar{h}$ the function $F_2(r)$ starts concave; whereas, if $h > \bar{h}$ it starts convex.
\begin{figure}[h!]
 \centering
\includegraphics[width=.45\textwidth]{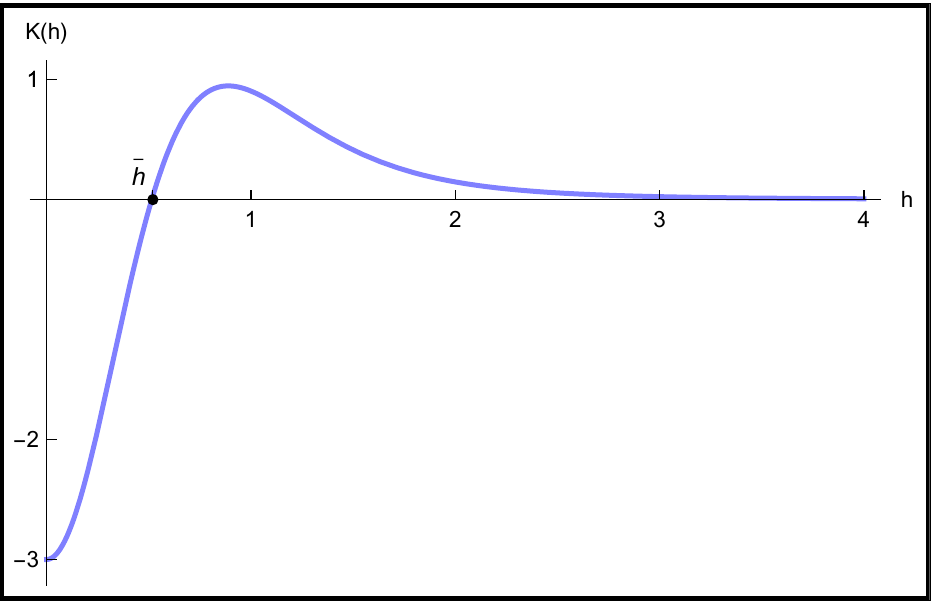}
\caption{Plot of the function $K(h)$. The value $\bar{h}$ is defined by the equation $K(h)=0$, which is equivalent to~\eqref{h:bar}.}
\label{fig:kappa}
\end{figure}
\end{itemize}

\bigskip
To conclude the proof it remains to show that $\theta_1(h) < \theta_2(h)$, for every $h > 0$. If $I_v(\cdot)$ denotes a first kind modified Bessel function of order $v$, we must prove that 
\[
1 - 2 \int_0^{2\pi} \cos^2 x \, q^{(0)} (x,+1) \, dx + \left( \int_0^{2\pi} \cos x \, q^{(0)} (x,+1) \, dx \right)^2 =\frac{I_1^2 (2h) - I_0 (2h) I_2 (2h)}{I_0^2 (2h)} 
\]
is strictly positive for all values of $h > 0$. The assertion follows from the inequality~\cite{JoBi91}
\begin{equation}\label{ineq:Bessel}
I_v^2 (y) - I_{v-n} (y) I_{v+n} (y) > 0, \quad \mbox{ whenever } v > 0, y > 0, n \geq 1.
\end{equation} 

\appendix

\section{Appendix}

\subsection{Derivation of formula \eqref{conv:Bessel}}\label{app:integral}

We devote this section to compute $\int_0^{2\pi} \cos x \, q (x, +1) \, dx$. To shorten our notation, let us introduce constants
\[
A := \theta r \cos \Psi + h \quad \mbox{ and } \quad B := \theta r \sin \Psi \,.
\] 
Therefore, we have
\[
Z(+1) \int_0^{2\pi} \cos x \, q (x, +1) \, dx = Z(+1) \int_0^{2\pi} \cos x \, \exp \left\{ 2 A \cos x + 2 B \sin x \right\} \, dx 
\]  
and, by using the power series expansion for the exponential function, the right-hand side can be expressed as a sum over odd terms (being the even ones zero)
\begin{equation}\label{series:1}
Z(+1) \int_0^{2\pi} \cos x \, q (x,+1) dx = \sum_{k=0}^{+\infty} \frac{2^{2k+1}}{(2k+1)!} \int_0^{2\pi}\cos x \, \left( A \cos x + B \sin x \right)^{2k+1} dx \,.
\end{equation}
By writing the trigonometric functions in terms of the complex exponential we can expand the powers of binomials to get
\begin{multline*}
\int_0^{2\pi}\cos x \, \left( A \cos x + B \sin x \right)^{2k+1} dx \\
= \frac{1}{2^{2(k+1)}} \sum_{j=0}^{2k+1} \sum_{h=0}^{j+1} \sum_{\ell=0}^{2k+1-j} (-1)^{\ell} {2k+1 \choose j} {j+1 \choose h} {2k+1-j \choose \ell} A^j (iB)^{2k+1-j} \int_0^{2\pi} e^{2(h+\ell-k-1)ix} \, dx \,.
\end{multline*}
Now observe that 
\begin{itemize}
\item[$\bullet$]  
the only non-zero terms in the triple sum are those for which $h + \ell - k - 1 = 0$;
\item[$\bullet$] 
$j$ must be odd for the whole sum to be real;
\end{itemize}
and thus, after the index change $j \to 2j+1$,
\begin{align*}
\int_0^{2\pi}\cos x \, ( A \cos x & + B \sin x )^{2k+1} dx  \\
&= \frac{2 \pi A}{2^{2(k+1)}} \sum_{j=0}^{k} \sum_{h=0}^{k+1}  (-1)^{j+h+1} {2k+1 \choose 2j+1} {2j+2 \choose h} {2k-2j \choose k+1-h} A^{2j} B^{2(k-j)}\,.
\end{align*}
To continue we need the following technical lemma.

\begin{lemma}\label{lmm:sum}
Let $m,n \in \mathbb{N}$, with $m > n$. Then, 
\[
\sum_{\ell=0}^n  {n \choose \ell} \frac{(-1)^{\ell+1}}{(m - \ell)! \, (m - n + \ell)!} = 
\begin{cases}
\frac{(-1)^{\frac{n}{2}+1} \, n (n-1) \cdots \left( \frac{n}{2}+1 \right)}{m! \, \left( m - \frac{n}{2} \right)!}& \mbox{ for $n$ even} \\
&\\
0 & \mbox{ for $n$ odd}. 
\end{cases}
\]
\end{lemma}

\begin{proof}
We use generating functions. Consider
\[
A(z) := \sum_{i=0}^{+\infty} a_i \frac{z^i}{i!} \qquad \mbox{ and } \qquad B(z) := \sum_{i=0}^{+\infty} b_i  \frac{z^i}{i!},
\]
then formally we have
\begin{equation}\label{prod:gen:fcts:formal}
A(z)B(z) = \sum_{n=0}^{+\infty} \left( \sum_{\ell=0}^n {n \choose \ell} a_{\ell} \, b_{n-\ell} \right) \frac{z^n}{n!} = \sum_{n=0}^{+\infty} c_n \frac{z^n}{n!} \qquad \mbox{ with } \qquad c_n :=  \sum_{\ell=0}^n {n \choose \ell} a_{\ell} \, b_{n-\ell}.
\end{equation}
We set
\[ 
a_i := \frac{(-1)^{i+1}}{(m-i)!} \qquad \mbox{ and } \qquad b_i := \frac{1}{(m-i)!}
\]
and we determine the generating functions as
\[
A(z) = \sum_{i=0}^{+\infty} \frac{(-1)^{i+1}}{(m-i)!} \frac{z^i}{i!} = - \frac{1}{m!} \sum_{i=0}^{+\infty} {m \choose i} (-z)^i = - \frac{(1-z)^m}{m!}
\]
and
\[
B(z) = \sum_{i=0}^{+\infty} \frac{1}{(m-i)!} \frac{z^i}{i!} =  \frac{1}{m!} \sum_{i=0}^{+\infty} {m \choose i} z^i =  \frac{(1+z)^m}{m!} \,.
\]
So for the product we obtain
\[
A(z)B(z) = - \frac{\left( 1-z^2 \right)^m}{(m!)^2} = - \frac{1}{(m!)^2} \sum_{n=0}^{+\infty} {m \choose n} (-1)^n \, z^{2n} \,,
\]
which is a power series comprised of even powers only and therefore can be rewritten as
\begin{equation}\label{prod:gen:fcts}
A(z)B(z) =  \sum_{n \in 2\mathbb{N}} \frac{n!}{(m!)^2} {m \choose \frac{n}{2}} (-1)^{\frac{n}{2}+1} \,\frac{z^n}{n!} \,.
\end{equation}
By comparing equations \eqref{prod:gen:fcts:formal} and \eqref{prod:gen:fcts} and equating the coefficients corresponding to powers of the same order we get the conclusion.\\
\end{proof}
By Lemma~\ref{lmm:sum}, setting $n=2j+2$ and $m=k+1$, we obtain
\[
{k \choose j}^{-1} \left[  \sum_{h=0}^{2j+2} (-1)^{h+1} {2j+2 \choose h} {2k-2j \choose k+1-h} \right] = \frac{2 (-1)^j (2j+1)! (2k-2j)!}{(k+1)! \, k!}
\]
that in turn implies
\begin{equation} \label{series:2}
\int_0^{2\pi}\cos x \, ( A \cos x + B \sin x )^{2k+1} dx = \frac{2 \pi A  (2k+1)!}{2^{2k+1} (k+1)! \, k!} \sum_{j=0}^{k} {k \choose j} A^{2j} B^{2(k-j)}.
\end{equation}
Plugging \eqref{series:2} into \eqref{series:1} yields
\[
Z(+1) \int_0^{2\pi} \cos x \, q (x,+1) dx = 2 \pi A \sum_{k=0}^{+\infty} \frac{\left( A^2 + B^2 \right)^{k}}{ (k+1)! \, k!} \,.
\]
The final formula \eqref{conv:Bessel} follows from the series representation 
\[
I_v (y) = \sum_{k=0}^{+\infty} \frac{\left( \frac{y}{2} \right)^{v+2k}}{(v+k)! \, k!} 
\]
of the first kind modified Bessel function of order~$v$.

\subsection{A technical lemma on Bessel functions}\label{app:Bessel}

We state and prove a technical lemma that is useful in the proofs of Proposition \ref{prop:r=0} and Proposition \ref{integralr}. As usual, let us denote by $I_v(\cdot)$ a first kind modified Bessel function of order $v$. Then,

\begin{lemma}\label{tech:lemma}
The function $g$ defined as 
\[
g(z) := \frac{I_1(2\sqrt{z})}{\sqrt{z} \, I_0(2\sqrt{z})}
\]
is strictly decreasing on $]0,+\infty[$.
\end{lemma}

\begin{proof}
The proof relies on the following properties: the recurrent relation~\cite[Sect. 3.71]{Wat44}
\begin{equation}\label{rec:rel:Bessel}
y I_v'(y) - v I_v(y) = y I_{v+1}(y)
\end{equation}
and the inequality \eqref{ineq:Bessel}. We obtain
\[
g'(z) = \frac{I_0(2\sqrt{z}) \big[ 2\sqrt{z} \, I_1'(2\sqrt{z}) - I_1(2\sqrt{z}) \big] - 2\sqrt{z} \, I_0'(2\sqrt{z}) I_1(2\sqrt{z})}{2 z \sqrt{z} \, I_0^2(2\sqrt{z})} \, \stackrel{\mbox{\tiny \eqref{rec:rel:Bessel}}}{=} \, \frac{I_0(2\sqrt{z}) I_2(2\sqrt{z}) - I_1^2(2\sqrt{z})}{z I_0^2(2\sqrt{z})} \, \stackrel{\mbox{\tiny \eqref{ineq:Bessel}}}{<} \, 0 
\]
and the conclusion follows.
\end{proof}

\section*{Acknowledgements}
The authors express their gratitude to Aernout C.D. van Enter  who has been a great source of inspiration as witnessed also by the content of this paper. Moreover they warmly thank Paolo Dai Pra, Marco Formentin, Frank Redig and Cristian Spitoni for interesting suggestions and discussions.\\
Research of FC was supported by the Italian research funding agency (MIUR) through FIRB research grant RBFR10N90W, by the French foundation \emph{Fondation Sciences Math{\'e}matiques de Paris} and by the Dutch stochastics cluster \emph{STAR} (Stochastics -- Theoretical and Applied Research).

\bibliographystyle{spmpsci}      

\end{document}